\documentclass[final,onefignum,onetabnum]{siamltex}

\usepackage{amssymb} 
\usepackage{amsfonts}
\usepackage{epsfig}

\newcommand{\be}{\begin{equation}}
\newcommand{\ee}{\end{equation}}

\newtheorem{thm}{Theorem}[section]
\newtheorem{lem}{Lemma}[section]
\newtheorem{dfn}{Definition}[section]

\title{GLOBAL SEARCH BASED ON EFFICIENT DIAGONAL PARTITIONS AND A SET OF
LIPSCHITZ CONSTANTS\thanks{This work has been partially supported
by the following grants: FIRB RBAU01JYPN, FIRB RBNE01WBBB, and
RFBR 04-01-00455-a.}}

\author{Yaroslav D. Sergeyev\thanks{Dipartimento di
Elettronica, Informatica e Sistemistica, Universit\`a della
Calabria, Via P.Bucci, Cubo 41C -- 87036 Rende (CS), Italy, and
N.I.~Lobachevski University of Nizhni Novgorod, Russia ({\tt
yaro@si.deis.unical.it})} \and Dmitri E.
Kvasov\thanks{Dipartimento di Statistica, Universit\`a di Roma
``La Sapienza'', P.le A.~Moro 5 -- 00185 Roma, Italy, and
N.I.~Lobachevski University of Nizhni Novgorod, Russia  ({\tt
kvadim@si.deis.unical.it})}}

\begin{document}

\maketitle

\begin{abstract}
In the paper, the global optimization problem of a
multidimensional ``black-box'' function satisfying the Lipschitz
condition over a hyperinterval with an unknown Lipschitz constant
is considered. A new efficient algorithm for solving this problem
is presented. At each iteration of the method a number of possible
Lipschitz constants is chosen from a set of values varying from
zero to infinity. This idea is unified with an efficient diagonal
partition strategy. A novel technique balancing usage of local and
global information during partitioning is proposed. A new
procedure for finding lower bounds of the objective function over
hyperintervals is also considered. It is demonstrated by extensive
numerical experiments performed on more than 1600 multidimensional
test functions that the new algorithm shows a very promising
performance.
\end{abstract}



\pagestyle{myheadings} \thispagestyle{plain}
\markboth{Ya.D.~Sergeyev, D.E.~Kvasov}{Global search based on
efficient diagonal partitions and a set of Lipschitz constants}

\section{Introduction}

Many decision-making problems arising in various fields of human
activity (technological processes, economic models, etc.) can be
stated as global optimization problems (see,
e.g.,~\cite{Floudas:et:al(1999), Pinter(1996),
Strongin&Sergeyev(2000)}). Objective functions describing
real-life applications are very often multiextremal,
non-differentiable, and hard to be evaluated. Numerical techniques
for finding solutions to such problems have been widely discussed
in literature (see, e.g.,~\cite{Dixon&Szego(1978),
Horst&Pardalos(1995), Horst&Tuy(1993), Pinter(1996),
Strongin&Sergeyev(2000)}).

In this paper, the Lipschitz global optimization problem is
considered. This type of optimization problems is sufficiently
general both from theoretical and applied points of view. In fact,
it is based on a rather natural assumption that any limited change
in the parameters of the objective function yields some limited
changes in the characteristics of the object's performance. The
knowledge of a bound on the rate of change of the objective
function, expressed by the Lipschitz constant, allows one to
construct global optimization algorithms and to prove their
convergence (see, e.g.,~\cite{Horst&Pardalos(1995),
Horst&Tuy(1993), Pinter(1996), Strongin&Sergeyev(2000)}).

Mathematically, the global optimization problem considered in the
paper can be formulated as minimization of a multidimensional
multiextremal ``black-box'' function that satisfies the Lipschitz
condition over a domain $D \subset \mathbb{R} ^N$ with an unknown
constant~$L$, i.e., finding the value $f^*$ and points $x^*$ such
that
 \be
  f^*=f(x^*)=\min_{x\in D } f(x), \label{(1.1)}
 \ee
 \be
  | f(x') - f(x'') | \leq L \| x'-x'' \| , \hspace{3mm} x',x'' \in
  D, \hspace{4mm} 0<L<\infty, \label{(1.2)}
 \ee
where
 \be
  D=[a,b\,]=\{x \in \mathbb{R} ^N : a(j)\leq x(j)\leq b(j), \, 1\leq j \leq N \},  \label{(1.3)}
 \ee
$a$, $b$ are given vectors in $\mathbb{R} ^N$, and $\| \cdot \|$
denotes the Euclidean norm.

The function $f(x)$ is supposed to be non-differentiable. Hence,
optimization methods using derivatives cannot be used for solving
problem (\ref{(1.1)})--(\ref{(1.3)}). It is also assumed that
evaluation of the objective function at a point (also referred to
as a `function trial') is a time-consuming operation.

Numerous algorithms have been proposed (see,
e.g.,~\cite{Dixon&Szego(1978), Horst&Pardalos(1995),
Horst&Tuy(1993), Jones:et:al.(1993), Kvasov&Sergeyev(2003),
Mladineo(1986), Pinter(1996), Piyavskij(1972), Sergeyev(1995a),
Shubert(1972), Strongin&Sergeyev(2000)}) for solving problem
(\ref{(1.1)})--(\ref{(1.3)}). One of the main questions to be
considered in this occasion is: How can the Lipschitz constant $L$
be specified? There are several approaches to specify the
Lipschitz constant. First of all, it can be given a priory (see,
e.g.,~\cite{Horst&Pardalos(1995), Horst&Tuy(1993), Mladineo(1986),
Piyavskij(1972), Shubert(1972)}). This case is very important from
the theoretical viewpoint but is not frequently encountered in
practice. The more promising and practical approaches are based on
an adaptive estimation of $L$ in the course of the search. In such
a way, algorithms can use either a global estimate of the
Lipschitz constant (see, e.g.,~\cite{Kvasov&Sergeyev(2003),
Pinter(1996), Strongin&Sergeyev(2000)}) valid for the whole region
$D$ from (\ref{(1.3)}), or local estimates $L_i$ valid only for
some subregions $D_i \subseteq D$ (see,
e.g.,~\cite{Kvasov:et:al.(2003), Molinaro:et:al.(2001),
Sergeyev(1995a), Sergeyev(1995b), Strongin&Sergeyev(2000)}).

Since the Lipschitz constant has a significant influence on the
convergence speed of the Lipschitz global optimization algorithms,
the problem of its specifying is of the great importance. In fact,
accepting too high value of $L$ for a concrete objective function
means assuming that the function has complicated structure with
sharp peaks and narrow attraction regions of minimizers within the
whole admissible region. Thus, too high value of $L$ (if it does
not correspond to the real behavior of the objective function)
leads to a slow convergence of the algorithm to the global
minimizer.

Global optimization algorithms using in their work a global
estimate of $L$ (or some value of $L$ given a priori) do not take
into account local information about behavior of the objective
function over every small subregion of $D$. As it has been
demonstrated in~\cite{Sergeyev(1995a), Sergeyev(1995b),
Strongin&Sergeyev(2000)}, estimating local Lipschitz constants
allows us to accelerate significantly the global search.
Naturally, balancing between local and global information must be
performed in an appropriate way to avoid the missing of the global
solution.

Recently, an interesting approach unifying usage of local and
global information during the global search has been proposed
in~\cite{Jones:et:al.(1993)}. At each iteration of this new
algorithm, called DIRECT, instead of only one estimate of the
Lipschitz constant a set of possible values of $L$ is used.

As many Lipschitz global optimization algorithms, DIRECT tries to
find the global minimizer by partitioning the search hyperinterval
$D$ into smaller hyperintervals $D_i$ using a particular partition
scheme described in~\cite{Jones:et:al.(1993)}. The objective
function is evaluated only at the central point of a
hyperinterval. Each hyperinterval $D_i$ of a current partition of
$D$ is characterized by a lower bound of the objective function
over this hyperinterval. It is calculated similarly
to~\cite{Piyavskij(1972), Shubert(1972)} taking into account the
Lipschitz condition~(\ref{(1.2)}). A hyperinterval $D_i$ is
selected for a further partitioning if and only if for some value
$\tilde{L}>0$ (which estimates the unknown constant $L$) it has
the smallest lower bound of $f(x)$ with respect to the other
hyperintervals. By changing~$\tilde{L}$ from zero to infinity, at
each iteration DIRECT selects several `potentially optimal'
hyperintervals (see~\cite{Gablonsky(2001), He:et:al.(2002),
Jones:et:al.(1993)}) in such a way that for a particular estimate
of the Lipschitz constant the objective function could have the
smallest lower bound over every potentially optimal hyperinterval.

Due to its simplicity and efficiency, DIRECT has been widely
adopted in practical applications (see,
e.g.,~\cite{Baker:et:al.(2001), Bartholomew-Biggs:et:al.(2002),
Carter:et:al.(2001), Cox:et:al.(2001), Gablonsky(2001),
He:et:al.(2002), Ljungberg:et:al.(1986), Watson&Baker(2001)}). In
fact, DIRECT is a derivative-free deterministic algorithm which
does not require multiply runs. It has only one parameter which is
easy to set (see~\cite{Jones:et:al.(1993)}). The center-sampling
partition strategy of DIRECT reduces the computational complexity
in high-dimensional spaces allowing DIRECT to demonstrate good
performance results (see~\cite{Gablonsky&Kelley(2001),
Jones:et:al.(1993)}).

However, some aspects which can limit the applications of DIRECT
have been pointed out by several authors (see,
e.g.,~\cite{Cox:et:al.(2001), Finkel&Kelley(2004),
He:et:al.(2002), Jones(2001)}). First of all, it is difficult to
apply for DIRECT some meaningful stopping criterion, such as, for
example, stopping on achieving a desired accuracy in solution.
This happens because DIRECT does not use a single estimate of the
Lipschitz constant but a set of possible values of~$L$. Although
several attempts of introducing a reasonable criterion of arrest
have been done (see, e.g.,~\cite{Bartholomew-Biggs:et:al.(2002),
Cox:et:al.(2001), Gablonsky(2001), He:et:al.(2002)}), termination
of the search process caused by exhaustion of the available
computing resources (such as maximal number of function trials)
remains the most interesting for practical engineering
applications.

Another important observation regarding DIRECT is related to the
partition and sampling strategies adopted by the algorithm
(see~\cite{Jones:et:al.(1993)}) which simplicity turns into some
problems. As it has been outlined in~\cite{Cox:et:al.(2001),
Finkel&Kelley(2004), Ljungberg:et:al.(1986)}, DIRECT is quick to
locate regions of local optima but slow to converge to the global
one. This can happen for several reasons. The first one is a
redundant (especially in high dimensions, see~\cite{Jones(2001)})
partition of hyperintervals along all longest sides. The next
cause of DIRECT's slow convergence can be excessive partition of
many small hyperintervals located in the vicinity of local
minimizers which are not global ones. Finally, DIRECT --- like all
center-sampling partitioning schemes --- uses a relatively poor
information about behavior of the objective function~$f(x)$. This
information is obtained by evaluating $f(x)$ only at one central
point of each hyperinterval without considering the adjacent
hyperintervals. Due to this fact, DIRECT can manifest slow
convergence (as it has been highlighted
in~\cite{Huyer&Neumaier(1999)}) in cases when the global minimizer
lies at the boundary of the admissible region $D$ from
(\ref{(1.3)}).

There are several modifications to the original DIRECT algorithm.
For example, in~\cite{Jones(2001)}, partitioning along only one
long side is suggested to accelerate convergence in high
dimensions. The problem of stagnation of DIRECT near local
minimizers (emphasized, e.g., in~\cite{Finkel&Kelley(2004)}) can
be attacked by changing the parameter of the algorithm
(see~\cite{Jones:et:al.(1993)}) preventing DIRECT from being too
local in its orientation (see~\cite{Finkel&Kelley(2004),
Gablonsky(2001), He:et:al.(2002), Jones(2001),
Jones:et:al.(1993)}). But in this case the algorithm becomes too
sensitive to tuning such a parameter, especially for difficult
``black-box'' global optimization problems
(\ref{(1.1)})--(\ref{(1.3)}). In~\cite{Baker:et:al.(2001),
Watson&Baker(2001)}, another modification to DIRECT, called
`aggressive DIRECT', has been proposed. It subdivides all
hyperintervals with the smallest function value for each
hyperinterval size. This results in more hyperintervals
partitioned at every iteration, but the number of hyperintervals
to be subdivided grows up significantly. In~\cite{Gablonsky(2001),
Gablonsky&Kelley(2001)}, the opposite idea which is more biased
toward local improvement of the objective function has been
studied. Results obtained in~\cite{Gablonsky(2001),
Gablonsky&Kelley(2001)} demonstrate that this modification seems
to be more suitable for low dimensional problems with a single
global minimizer and a few local minimizers.

The goal of this paper is to present a new algorithm which would
be oriented (in contrast with the algorithm
from~\cite{Gablonsky(2001), Gablonsky&Kelley(2001)}) on solving
`difficult' multidimensional multiextremal ``black-box'' problems
(\ref{(1.1)})--(\ref{(1.3)}). It uses a new technique for
selection of hyperintervals to be subdivided which is unified with
a new diagonal partition strategy. A new procedure for estimation
of lower bounds of the objective function over hyperintervals is
combined with the idea (introduced in DIRECT) of usage of a set of
Lipschitz constants instead of a unique estimate. As demonstrated
by extensive numerical results, application of the new algorithm
to minimizing hard multidimensional ``black-box'' functions leads
to significant improvements.

The paper is organized as follows. In Section~2, a theoretical
background of the new algorithm --- a new partition strategy, a
new technique for lower bounding of the objective function over
hyperintervals, and a procedure for selection of `non-dominated'
hyperintervals for eventual partitioning --- is presented.
Section~3 is dedicated to description of the algorithm and to its
convergence analysis. Finally, Section~4 contains results of
numerical experiments executed on more than 1600 test functions.

\section{Theoretical background} \label{sectionTheory}

This section consists of the following three parts. First, a new
partition strategy developed in the framework of diagonal approach
is described. The second part presents a new procedure for
estimation of lower bounds of the objective function over
hyperintervals. The third part is dedicated to description of a
procedure for determining non-dominated hyperintervals ---
hyperintervals that have the smallest lower bound for some
particular estimate of the Lipschitz constant.

\subsection{Partition strategy} \label{sectionNewStrategy}

In global optimization algorithms, various techniques for adaptive
partition of the admissible region $D$ into a set of
hyperintervals $D_i$ are used (see, e.g.,
\cite{Horst&Pardalos(1995), Jones:et:al.(1993), Pinter(1996),
Sergeyev(2000)}) for solving (\ref{(1.1)})--(\ref{(1.3)}). A
current partition $\{D^k\}$ of $D$ in the course of an iteration
$k \geq 1$ of an algorithm can be represented as
 \be
  D=\cup_{i=1}^{m(k)+\Delta m(k)}D_i, \ D_i \cap D_j = \delta (D_i) \cap
  \delta (D_j), \ i \neq j. \label{PartitionD}
 \ee
Here, $\delta(D_i)$ denotes the boundary of $D_i$, $m(k)$ is the
number of hyperintervals at the beginning of the iteration $k$,
and $\Delta m(k)$ is the current number of new hyperintervals
produced during the $k$-th iteration. For example, if only one new
hyperinterval is generated at every iteration then $\Delta
m(k)=1$.

Over each hyperinterval $D_i \in \{D^k\}$, approximation of $f(x)$
is based on results obtained by evaluating $f(x)$ at some points
$x \in D$. For example, DIRECT~\cite{Jones:et:al.(1993)} involves
partitioning with evaluation of $f(x)$ at the central points of
hyperintervals (note that for DIRECT the number $\Delta m(k)$
in~(\ref{PartitionD}) can be greater than 1).

In this paper, the diagonal approach proposed
in~\cite{Pinter(1986), Pinter(1996)} is considered. In this
approach, the function $f(x)$ is evaluated only at two vertices
$a_i$ and $b_i$ of the main diagonals of each hyperinterval~$D_i$
independently of the problem dimension (recall that each
evaluation of $f(x)$ is a time-consuming operation).

Among attractions of the diagonal approach there are the following
two. First, the objective function is evaluated at two points at
each hyperinterval. Thus, diagonal algorithms obtain a more
complete information about the objective function than
central-sampling methods. Second, many efficient one-dimensional
global optimization algorithms can be easily extended to the
multivariate case by means of the diagonal scheme (see, e.g,
\cite{Kvasov:et:al.(2003), Kvasov&Sergeyev(2003),
Molinaro:et:al.(2001), Pinter(1986), Pinter(1996)}).

As shown in \cite{Kvasov&Sergeyev(2003), Sergeyev(2000)}, diagonal
global optimization algorithms based on widely used partition
strategies (such as bisection or partition $2^N$ used
in~\cite{Pinter(1986), Pinter(1996)}) produce many redundant
trials of the objective function. This redundancy slows down the
algorithm execution because of high time required for the
evaluations of $f(x)$. It also increases the computer memory
allocated for storing the redundant information.

The new partition strategy proposed in \cite{Sergeyev(2000)} (see
also \cite{Kvasov&Sergeyev(2003)}) overcomes these drawbacks of
conventional diagonal partition strategies. We start its
description by a two-dimensional example in Fig.~\ref{fig_ADC}. In
this figure, partitions of the admissible region~$D$ produced by
the algorithm at the initial iterations are presented. We suppose
just for simplicity that at each iteration only one hyperinterval
can be subdivided. Trial points of~$f(x)$ are represented by black
dots. The numbers around these dots indicate iterations at which
the objective function is evaluated at the corresponding points.
The terms `interval' and `subinterval' will be used to denote
two-dimensional rectangular domains.

In Fig.~\ref{fig_ADC}a, situation after the first two iterations
is presented. At the first iteration, the objective function
$f(x)$ is evaluated at the vertices $a$ and $b$ of the search
domain $D=[a,b]$. At the next iteration, the interval $D$ is
subdivided into three subintervals of equal area (equal volume in
general case). This subdivision is performed by two lines
(hyperplanes) orthogonal to the longest edge of $D$ and passing
through points $u$ and $v$ (see Fig.~\ref{fig_ADC}a). The
objective function is evaluated at both points $u$ and~$v$.

\begin{figure}[t]
\centerline{\epsfig{file=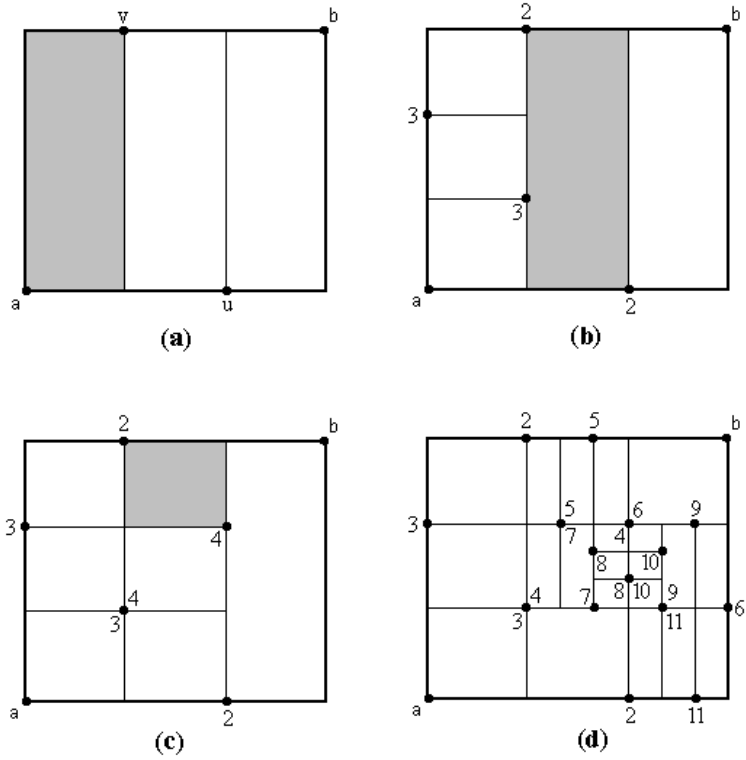,width=110mm,height=110mm,angle=0,silent=}}
\vspace{-7mm} \caption{An example of subdivisions by a new
partition strategy.} \label{fig_ADC}
\end{figure}

Suppose that the interval shown in grey in Fig.~\ref{fig_ADC}a is
chosen for the further partitioning. Thus, at the third iteration,
three smaller subintervals appear (see Fig.~\ref{fig_ADC}b). It
seems that a trial point of the third iteration is redundant for
the interval (shown in grey in Fig.~\ref{fig_ADC}b) selected for
the next splitting. But in reality, Fig.~\ref{fig_ADC}c
demonstrates that one of the two points of the fourth iteration
coincides with the point 3 at which $f(x)$ has already been
evaluated. Therefore, there is no need to evaluate $f(x)$ at this
point again, since the function value obtained at the previous
iteration can be used. This value can be stored in a special
vertex database and is simply retrieved when it is necessary
without reevaluation of the function. For example,
Fig.~\ref{fig_ADC}d illustrates situation after 11 iterations.
Among 22 points at which the objective function is to be
evaluated, there are 5 repeated points. That is, $f(x)$ is
evaluated 17 rather than 22 times. Note also that the number of
generated intervals (equal to~21) is greater than the number of
trial points (equal to~17). Such a difference becomes more
pronounced in the course of further subdivisions
(see~\cite{Kvasov&Sergeyev(2003)}).

Let us now describe the general procedure of hyperinterval
subdivision. Without loss of generality, hereafter we assume that
the admissible region $D$ in~(\ref{(1.3)}) is an $N$-dimensional
hypercube. Suppose that at the beginning of an iteration $k \geq
1$ of the algorithm the current partition $\{D^k\}$ of $D=[a,b]$
consists of $m(k)$ hyperintervals and $\Delta m(k) \geq 0$ new
hyperintervals have been already obtained. Let a hyperinterval
$D_t=[a_t, b_t]$ be chosen for partitioning too. The operation of
partitioning the selected hyperinterval~$D_t$ is performed as
follows (we omit the iteration index in the formulae).

\vspace{2mm}

\begin{description}
 \item[{\bf Step 1.}] Determine points $u$ and $v$ by the following formulae
 \be
  u=(a(1),\ldots,a(i-1),a(i) + \frac{2}{3}
  (b(i)-a(i)),a(i+1),\ldots,a(N)), \label{partition:u}
 \ee
 \be
  v=(b(1),\ldots,b({i-1}),b(i) + \frac{2}{3}
  (a(i)-b(i)),b(i+1),\ldots,b(N)),\label{partition:v}
 \ee
where $a(j)=a_t(j),\ b(j)=b_t(j),\ 1 \leq j \leq N$, and $i$ is
given by the equation
 \be
   i = \arg \min\, \max_{1 \leq j \leq N} | b(j) - a(j) |.
   \label{side_i}
 \ee
Get (evaluate or read from  the vertex database) the values of the
objective function $f(x)$ at the points $u$ and $v$.

 \item[{\bf Step 2.}] Divide the hyperinterval $D_t$ into three
hyperintervals of equal volume by two parallel hyperplanes that
are perpendicular to the longest edge $i$ of $D_t$ and pass
through the points $u$ and $v$.

\hspace{2mm}The hyperintervals $D_t$ is so substituted by three
new hyperintervals with indices $t'=t$, $m + \Delta m+1$, and $m +
\Delta m+2$ determined by the vertices of their main diagonals
 \be \label{partition:D_t}
   a_{t'} = a_{m+\Delta m +2} = u, \  b_{t'} = b_{m+\Delta m+1} = v,
 \ee
 \be \label{partition:D_1}
   a_{m+\Delta m+1} = a_t, \ b_{m+\Delta m+1} = v,
 \ee
 \be \label{partition:D_2}
   a_{m+\Delta m+2} = u, \  b_{m+\Delta m+2} = b_t.
 \ee

 \item[{\bf Step 3.}] Augment the number of hyperintervals generated
 during the iteration $k$:
 \be \label{partition:Delta_m}
  \Delta m = \Delta m(k) := \Delta m(k) + 2.
 \ee

\end{description}

The existence of a special indexing of hyperintervals establishing
links between hyperintervals generated at different iterations has
been theoretically demonstrated in~\cite{Sergeyev(2000)}. It
allows one to store information about vertices and the
corresponding values of $f(x)$ in a special database avoiding in
such a manner redundant evaluations of~$f(x)$. The objective
function value at a vertex is calculated only once, stored in the
database, and read when required. The new partition strategy
generates trial points in such a way that one vertex where $f(x)$
is evaluated can belong to several (up to~$2^N$) hyperintervals
(see, for example, a trial point at the 8-th iteration in
Fig.~\ref{fig_ADC}d). Since the time-consuming operation of
evaluation of the function is replaced by a significantly faster
operation of reading (up to $2^N$ times) the function values from
the database, the new partition strategy considerably speeds up
the search and also leads to saving computer memory. It is
particularly important that the advantage of the new scheme
increases with the growth of the problem dimension
(see~\cite{Kvasov&Sergeyev(2003), Sergeyev(2000)}).

The new strategy can be viewed also as a procedure generating a
new type of space-filling curve
--- adaptive diagonal curve. This curve is constructed on the main
diagonals of hyperintervals obtained during subdivision of $D$.
The objective function is approximated over the multidimensional
region $D$ by evaluating $f(x)$ at the points of one-dimensional
adaptive diagonal curve. The order of partition of this curve is
different within different subintervals of~$D$. If selection of
hyperintervals for partitioning is realized appropriately in an
algorithm, the curve condenses in the vicinity of the global
minimizers of $f(x)$ (see~\cite{Kvasov&Sergeyev(2003),
Sergeyev(2000)}).

\subsection{Lower bounds}

Let us suppose that at some iteration $k > 1$ of the global
optimization algorithm the admissible region $D$ has been
partitioned into hyperintervals~$D_i \in \{D^k\}$ defined by their
main diagonals $[a_i, b_i]$. At least one of these hyperintervals
should be selected for further partitioning. In order to make this
selection, the algorithm estimates the goodness (or, in other
words, characteristics) of the generated hyperintervals with
respect to the global search. The best (in some predefined sense)
characteristic obtained over some hyperinterval $D_t$ corresponds
to a higher possibility to find the global minimizer within $D_t$.
This hyperinterval is subdivided at the next iteration of the
algorithm. Naturally, more than one `promising' hyperinterval can
be partitioned at every iteration.

One of the possible characteristics of a hyperinterval can be an
estimate of the lower bound of $f(x)$ over this hyperinterval.
Once all lower bounds for all hyperintervals of the current
partition~$\{D^k\}$ have been calculated, the hyperinterval with
the smallest lower bound can be selected for the further
partitioning.

Different approaches to finding lower bounds of $f(x)$ have been
proposed in literature (see, e.g., \cite{Horst&Pardalos(1995),
Jones:et:al.(1993), Mladineo(1986), Pinter(1996), Piyavskij(1972),
Shubert(1972), Strongin&Sergeyev(2000)}) for solving problem
(\ref{(1.1)})--(\ref{(1.3)}). For example, given the Lipschitz
constant~$L$, in~\cite{Horst&Pardalos(1995), Mladineo(1986),
Piyavskij(1972)} a support function for $f(x)$ is constructed as
the upper envelope of a set of $N$-dimensional circular cones of
the slope~$L$. Trial points of $f(x)$ are coordinates of the
vertices of the cones. At each iteration, the global minimizer of
the support function is determined and chosen as a new trial
point. Finding such a point requires analyzing the intersections
of all cones and, generally, is a difficult and time-consuming
task, especially in high dimensions.

If a partition of $D$ into hyperintervals is used, each cone can
be considered over the corresponding hyperinterval, independently
from the other cones. This allows one (see, e.g.,
\cite{Horst&Pardalos(1995), Jones:et:al.(1993),
Kvasov:et:al.(2003), Pinter(1996)}) to avoid the necessity of
establishing the intersections of the cones and to simplify the
lower bound estimation. For example, multidimensional DIRECT
algorithm \cite{Jones:et:al.(1993)} uses one cone with symmetry
axis passed through a central point of a hyperinterval for lower
bounding $f(x)$ over this hyperinterval. The lower bound is
obtained on the boundary of the hyperinterval. This approach is
simple, but it gives a too rough estimate of the minimum function
value over the hyperinterval.

The more accurate estimate is achieved when two trial points over
a hyperinterval are used for constructing a support function for
$f(x)$. These points can be, for example, the vertices $a_i$ and
$b_i$ of the main diagonal of a hyperinterval $D_i \in \{D^k\}$
(see, e.g., \cite{Kvasov:et:al.(2003), Kvasov&Sergeyev(2003),
Pinter(1986), Pinter(1996), Sergeyev(2000)}). In this case, the
objective function (due to the Lipschitz condition (\ref{(1.2)}))
must lie above the intersection of the $N$-dimensional cones
$C_1(x,L)$ and $C_2(x,L)$ (see the two-dimensional example in
Fig.~\ref{fig_Estimate}a). These cones have the slope~$L$ and are
limited by the boundaries of the hyperinterval $D_i$. The vertices
of the cones (in $(N+1)$-dimensional space) are defined by the
coordinates $(a_i, f(a_i))$ and $(b_i, f(b_i))$, respectively. In
such a way, the lower bound of $f(x)$ is more precise with respect
to the central-sampling strategy. Algorithms using this approach
are called diagonal (see, e.g., \cite{Kvasov:et:al.(2003),
Kvasov&Sergeyev(2003), Pinter(1986), Pinter(1996),
Sergeyev(2000)}).

\begin{figure}[t]
\centerline{\epsfig{file=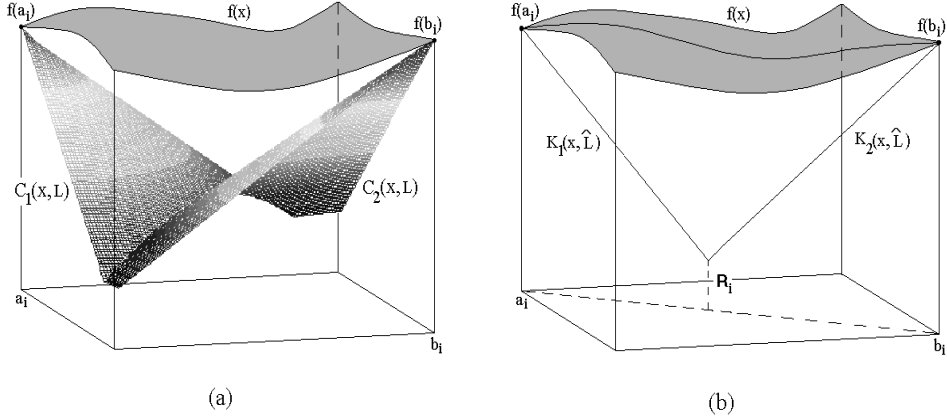,width=125mm,height=55mm,angle=0,silent=}}
\vspace{-1mm}\caption{Estimation of the lower bound of $f(x)$ over
an interval $D_i=[a_i, b_i]$.} \label{fig_Estimate}
\end{figure}

In the new diagonal algorithm proposed in this paper, the
objective function is also evaluated at two points of a
hyperinterval $D_i=[a_i, b_i]$. Instead of constructing a support
function for $f(x)$ over the whole hyperinterval $D_i$, we use a
support function for $f(x)$ only over the one-dimensional segment
$[a_i, b_i]$. This support function is the maximum of two linear
functions $K_1(x,\hat{L})$ and $K_2(x,\hat{L})$ passing with the
slopes~$\pm\hat{L}$ through the vertices $a_i$ and $b_i$ (see
Fig.~\ref{fig_Estimate}b). The lower bound of $f(x)$ over the
diagonal~$[a_i, b_i]$ of $D_i$ is calculated similarly to
\cite{Piyavskij(1972), Shubert(1972)} at the intersection of the
lines $K_1(x,\hat{L})$ and $K_2(x,\hat{L})$ and is given by the
following formula (see \cite{Kvasov:et:al.(2003), Pinter(1986),
Pinter(1996)})
 \be
   R_i = R_i(\hat{L}) = \frac{1}{2} (f(a_i)+f(b_i) - \hat{L} \|b_i - a_i \| ),
   \hspace{5mm} 0 < L \leq \hat{L} < \infty. \label{R_i}
 \ee

A valid estimate of the lower bound of $f(x)$ over $D_i$ can be
obtained from~(\ref{R_i}) if an overestimate $\hat{L}$ of the
Lipschitz constant $L$ is used. As it has been proved in
\cite{Pinter(1986), Pinter(1996)}, inequality
 \be
   \hat{L} \geq 2L \label{R(i)_Pinter}
 \ee
guarantees that the value $R_i$ from (\ref{R_i}) is the lower
bound of $f(x)$ over the whole hyperinterval $D_i$. Thus, the
lower bound of the objective function over the whole hyperinterval
$D_i \subseteq D$ can be estimated considering~$f(x)$ only along
the main diagonal $[a_i, b_i]$ of $D_i$.

A more precise than~(\ref{R(i)_Pinter}) condition ensuring that
 $$
   R_i(\hat{L}) \leq f(x), \hspace{5mm} x \in D_i,
 $$
is proved in the following theorem.

\begin{thm} \label{Theorem1}
Let $L$ be the known Lipschitz constant for $f(x)$ from
(\ref{(1.2)}), $D_i = [a_i, b_i]$ be a hyperinterval of a current
partition $\{D^k\}$, and $f^*_i$ be the minimum function value
over~$D_i$, i.e.,
 \be
  f^*_i=f(x^*_i), \hspace{5mm} x^*_i = \arg\min_{x \in D_i} f(x). \label{x*_on_D(i)}
 \ee
If an overestimate $\hat{L}$ in (\ref{R_i}) satisfies inequality
 \be
   \hat{L} \geq \sqrt{2}L, \label{R(i)_improved}
 \ee
then $R_i(\hat{L})$ from~(\ref{R_i}) is the lower bound of $f(x)$
over~$D_i$, i.e., $R_i(\hat{L}) \leq f^*_i$.
\end{thm}

{\em Proof.} Since $x^*_i$ from~(\ref{x*_on_D(i)}) belongs to
$D_i$ and $f(x)$ satisfies the Lipschitz condition~(\ref{(1.2)})
over~$D_i$, then the following inequalities hold
 $$
   f(a_i) - f^*_i \leq L \| a_i-x^*_i \|,
 $$
 $$
   f(b_i) - f^*_i \leq L \| b_i-x^*_i \|.
 $$
By summarizing these inequalities and using the following result
from~\cite[Lemma 2]{Molinaro:et:al.(2001)}
 $$
  \max_{x \in D_i}(\| a_i-x \| + \|b_i -x \|) \leq
  \sqrt{2}\| b_i-a_i \|,
 $$
we obtain
 $$
  f(a_i)+f(b_i) \leq 2f^*_i + L(\| a_i-x^*_i\| + \|b_i -x^*_i\|)\leq
 $$
 $$
  \leq 2f^*_i + L \max_{x \in D_i}(\| a_i-x \| + \|b_i -x \|)
  \leq 2f^*_i + \sqrt{2}L\| b_i-a_i \|.
 $$
Then, from the last inequality and (\ref{R(i)_improved}) we can
deduce that the following estimate holds for the value $R_i$ from
(\ref{R_i})
 $$
  R_i(\hat{L}) \leq \frac{1}{2}(2f^*_i + \sqrt{2}L \|b_i-a_i\| -
  \hat{L} \|b_i-a_i\|) =
 $$
 $$
  = f^*_i + \frac{1}{2}\underbrace{(\sqrt{2}L -\hat{L})}_{\leq 0}\|b_i - a_i \|
  \leq f^*_i. \qquad\endproof
 $$

Theorem~\ref{Theorem1} allows us to obtain a more precise lower
bound $R_i$ with respect to \cite{Pinter(1986), Pinter(1996)}
where estimate (\ref{R(i)_Pinter}) is considered.

\subsection{Finding non-dominated hyperintervals} \label{sectionNonDominated}

Let us now consider a diagonal partition $\{D^k\}$ of the
admissible region $D$, generated by the new subdivision strategy
from Section~\ref{sectionNewStrategy}. Let a positive value
$\tilde{L}$ be chosen as an estimate of the Lipschitz constant $L$
from~(\ref{(1.2)}) and lower bounds $R_i(\tilde{L})$ of the
objective function over hyperintervals $D_i \in \{D^k\}$ be
calculated by formula~(\ref{R_i}). Using the obtained lower bounds
of $f(x)$, the relation of domination can be established between
every two hyperintervals of a current partition $\{D^k\}$ of $D$.

\begin{dfn} \label{DefNonDominationL}
Given an estimate $\tilde{L}>0$ of the Lipschitz constant~$L$
from~(\ref{(1.2)}), a hyperinterval $D_i \in \{D^k\}$ {\rm
dominates} a hyperinterval $D_j \in \{D^k\}$ with respect to
$\tilde{L}$ if
$$
 R_i(\tilde{L}) < R_j(\tilde{L}).
$$

A hyperinterval $D_t \in \{D^k\}$ is said {\rm non-dominated with
respect to} $\tilde{L}>0$ if for the chosen value $\tilde{L}$
there is no other hyperinterval in~$\{D^k\}$ which dominates
$D_t$.
\end{dfn}

Each hyperinterval $D_i=[a_i, b_i] \in \{D^k\}$ can be represented
by a dot in a two-dimensional diagram (see Fig.~\ref{fig_Diagram})
similar to that used in DIRECT for representing hyperintervals
with $f(x)$ evaluated only at one point. The horizontal coordinate
$d_i$ and the vertical coordinate $F_i$ of the dot are defined as
follows
 \be \label{DiagramCoords}
   d_i = \frac{\| b_i - a_i \|}{2}, \hspace{3mm}
   F_i = \frac{f(a_i) + f(b_i)}{2}, \hspace{3mm} a_i \neq b_i.
 \ee

Note that a point ($d_i, F_i$) in the diagram can correspond to
several hyperintervals with the same length of the main diagonals
and the same sum of the function values at their vertices.

\begin{figure}[t]
\centerline{\epsfig{file=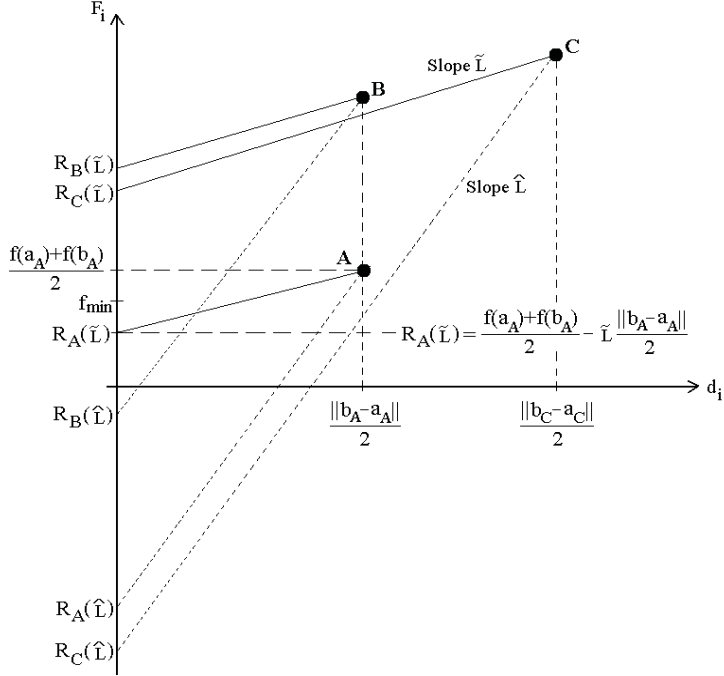,width=95mm,height=90mm,angle=0,silent=}}
\caption{Graphical interpretation of lower bounds of $f(x)$ over
hyperintervals.} \label{fig_Diagram}
\end{figure}

For the sake of illustration, let us consider a hyperinterval
$D_A$ with the main diagonal $[a_A, b_A]$. This hyperinterval is
represented by the dot~$A$ in Fig.~\ref{fig_Diagram}. Assuming an
estimate of the Lipschitz constant equal to $\tilde{L}$ (such that
condition~(\ref{R(i)_improved}) is satisfied), a lower bound
of~$f(x)$ over the hyperinterval $D_A$ is given by the value
$R_A(\tilde{L})$ from~(\ref{R_i}). This value is the vertical
coordinate of the intersection point of the line passed through
the point $A$ with the slope $\tilde{L}$ and the vertical
coordinate axis (see Fig.~\ref{fig_Diagram}). In fact, as it can
be seen from~(\ref{R_i}), intersection of the line with the slope
$\tilde{L}$ passed through any dot representing a hyperinterval in
the diagram of Fig.~\ref{fig_Diagram} and the vertical coordinate
axis gives us the lower bound~(\ref{R_i}) of $f(x)$ over the
corresponding hyperinterval.

Note that the points on the vertical axis ($d_i=0$) do not
represent any hyperinterval. The axis is used to express such
values as lower bounds, the current minimum value of the function,
etc. It should be highlighted that the current best value
$f_{min}$ is always smaller than or equal to the vertical
coordinate of the lowest dot (dot $A$ in Fig.~\ref{fig_Diagram}).
Note also that the vertex at which this value has been obtained
can belong to a hyperinterval, different from that represented by
the lowest dot in the diagram.

By using this graphical representation, it is easy to determine
whether a hyperinterval dominates (with respect to a given
estimate of the Lipschitz constant) some other hyperinterval from
a partition~$\{D^k\}$. For example, for the estimate $\tilde{L}$
the following inequalities are satisfied
(see~Fig.~\ref{fig_Diagram})
 $$
  R_A(\tilde{L}) < R_C(\tilde{L}) < R_B(\tilde{L}).
 $$
Therefore, with respect to $\tilde{L}$ the hyperinterval $D_A$
(dot $A$ in Fig.~\ref{fig_Diagram}) dominates both hyperintervals
$D_B$ (dot $B$) and $D_C$ (dot $C$), while $D_C$ dominates $D_B$.
If our partition~$\{D^k\}$ consists only of these three
hyperintervals, the hyperinterval $D_A$ is non-dominated with
respect to $\tilde{L}$.

If a higher estimate $\hat{L} > \tilde{L}$ of the Lipschitz
constant is considered (see Fig.~\ref{fig_Diagram}), the
hyperinterval $D_A$ still dominates the hyperinterval $D_B$ with
respect to $\hat{L}$, since $R_A (\hat{L}) < R_B(\hat{L})$. But
$D_A$ in its turn is dominated by the hyperinterval~$D_C$ with
respect to $\hat{L}$, because $R_A(\hat{L}) > R_C (\hat{L})$ (see
Fig.~\ref{fig_Diagram}). Thus, for the chosen estimate~$\hat{L}$
the unique non-dominated hyperinterval with respect to $\hat{L}$
is $D_C$, and not $D_A$ as previously.

As we can see from this simple example, some hyperintervals (as
the hyperinterval~$D_B$ in Fig.~\ref{fig_Diagram}) are always
dominated by another hyperintervals, independently on the chosen
estimate of the Lipschitz constant $L$. The following result
formalizing this fact takes place.

 \begin{lem} \label{Lemma1}
Given a partition $\{D^k\}$ of $D$ and the subset $\{D^k\}_d$ of
hyperintervals having the main diagonals equal to $d>0$, for any
estimate $\tilde{L}>0$ of the Lipschitz constant a hyperinterval
$D_t \in \{D^k\}_d$ dominates a hyperinterval $D_j \in \{D^k\}_d$
if and only if
 \be \label{ConditionLemma1}
  F_t = \min \{F_i: \hspace{2mm} D_i \in \{D^k\}_d \, \} < F_j,
 \ee
where $F_i$ and $F_j$ are from~(\ref{DiagramCoords}).
 \end{lem}
 \begin{proof}
The lemma follows immediately from (\ref{R_i}) since all
hyperintervals under consideration have the same length of their
main diagonals, i.e., $\| b_i - a_i \| = d$. \qquad
 \end{proof}

There also exist hyperintervals (for example, the hyperintervals
$D_A$ and $D_C$ represented in Fig.~\ref{fig_Diagram} by the dots
$A$ and $C$, respectively) that are non-dominated with respect to
one estimate of the Lipschitz constant $L$ and dominated with
respect to another estimate of $L$. Since in practical
applications the exact Lipschitz constant (or its valid
overestimate) is often unknown, the following idea inspired by
DIRECT~\cite{Jones:et:al.(1993)} is adopted.

At each iteration $k>1$ of the new algorithm, various estimates of
the Lipschitz constant~$L$ from zero to infinity are chosen for
lower bounding $f(x)$ over hyperintervals. The lower bound of
$f(x)$ over a particular hyperinterval is calculated by
formula~(\ref{R_i}). Note that since all possible values of the
Lipschitz constant are considered, condition~(\ref{R(i)_improved})
is automatically satisfied and no additional multipliers are
required for an estimate of the Lipschitz constant in~(\ref{R_i}).
Examination of the set of possible estimates of the Lipschitz
constant leads us to the following definition.

\begin{dfn} \label{DefNonDomination}
A hyperinterval $D_t \in \{D^k\}$ is called {\rm non-dominated} if
there exists an estimate $0 < \tilde{L} <\infty$ of the Lipschitz
constant~$L$ such that $D_t$ is non-dominated with respect to
$\tilde{L}$.
\end{dfn}

In other words, non-dominated hyperintervals are hyperintervals
over which $f(x)$ has the smallest lower bound for some particular
estimate of the Lipschitz constant. For example, in
Fig.~\ref{fig_Diagram} the hyperintervals $D_A$ and $D_C$ are
non-dominated.

Let us now make some observations that allow us to identify the
set of non-dominated hyperintervals. First of all, only
hyperintervals $D_t$ satisfying condition (\ref{ConditionLemma1})
can be non-dominated. In two-dimensional diagram $(d_i, F_i)$,
where $d_i$ and $F_i$ are from~(\ref{DiagramCoords}), such
hyperintervals are located at the bottom of each group of points
with the same horizontal coordinate, i.e., with the same length of
the main diagonals. For example, in Fig.~\ref{fig_Optimal} these
points are designated as $A$ (the largest interval), $B$, $C$,
$E$, $F$, $G$, and $H$ (the smallest interval).

\begin{figure}[t]
\centerline{\epsfig{file=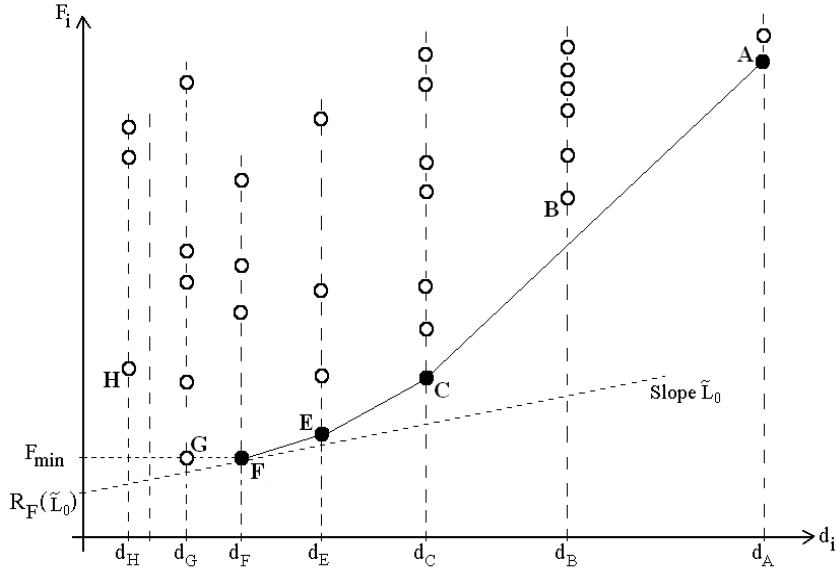,width=110mm,height=75mm,angle=0,silent=}}
\vspace{-2 mm}\caption{Dominated hyperintervals are represented by
white dots and non-dominated hyperintervals are represented by
black dots.} \label{fig_Optimal}
\end{figure}

It is important to notice that not all hyperintervals
satisfying~(\ref{ConditionLemma1}) are non-dominated. For example
(see Fig.~\ref{fig_Optimal}), the hyperinterval $D_H$ is dominated
(with respect to any positive estimate of the Lipschitz constant
$L$) by any of the hyperintervals $D_G$, $D_F$, or $D_E$. The
hyperinterval $D_G$ is dominated by $D_F$. In fact, as it follows
from~(\ref{R_i}), among several hyperintervals with the same sum
of the function values at their vertices, larger hyperintervals
dominate smaller ones with respect to any positive estimate of
$L$. Finally, the hyperinterval $D_B$ is dominated either by the
hyperinterval~$D_A$ (for example, with respect to $\tilde{L}_1
\geq \tilde{L}_{AC}$, where $\tilde{L}_{AC}$ corresponds to the
slope of the line passed through the points $A$ and $C$ in
Fig.~\ref{fig_Optimal}), or by the hyperinterval $D_C$ (with
respect to $\tilde{L}_2 < \tilde{L}_{AC}$).

Note that if an estimate $\tilde{L}$ of the Lipschitz constant is
chosen, it is easy to indicate the hyperinterval with the smallest
lower bound of $f(x)$, that is, non-dominated hyperinterval with
respect to $\tilde{L}$. To do this, it is sufficient to position a
line with the slope $\tilde{L}$ below the set of dots in
two-dimensional diagram representing hyperintervals of $\{D^k\}$,
and then to shift it upwards. The first dot touched by the line
indicates the desirable hyperinterval. For example, in
Fig.~\ref{fig_Optimal} the hyperinterval $D_F$ represented by the
point~$F$ is non-dominated hyperinterval with respect to
$\tilde{L}_0$, since over this hyperinterval $f(x)$ has the
smallest lower bound $R_F(\tilde{L}_0)$ for the given estimate
$\tilde{L}_0$ of the Lipschitz constant.

Let us now examine various estimates of the Lipschitz constant $L$
from zero to infinity. When a small (close to zero) positive
estimate of $L$ is chosen, an almost horizontal line is considered
in two-dimensional diagram representing hyperintervals of a
partition~$\{D^k\}$. The dot with the smallest vertical coordinate
$F_{min}$ (and the biggest horizontal coordinate if there are
several such dots) is the first to be touched by this line (the
case of the dot $F$ in Fig.~\ref{fig_Optimal}). Therefore,
hyperinterval (or hyperintervals) represented by this dot is
non-dominated with respect to the chosen estimate of $L$, and,
consequently, non-dominated in sense of
Def.~\ref{DefNonDomination}. Repeating such a procedure with
higher estimates of the Lipschitz constant (that is, considering
lines with higher slopes), all non-dominated hyperintervals can be
identified. In Fig.~\ref{fig_Optimal} the hyperintervals
represented by the dots $F$, $E$, $C$, and $A$ are non-dominated
hyperintervals.

This procedure can be formalized in terms of the algorithm known
as Jarvis march (or gift wrapping; see, e.g.,
\cite{Preparata&Shamos(1993)}), which is an algorithm for
identifying the convex hull of the dots. Thus, the following
result identifying the set of non-dominated hyperintervals for a
given partition~$\{D^k\}$ has been proved.

\begin{thm} \label{Theorem2}
Let each hyperinterval $D_i=[a_i, b_i] \in \{D^k\}$ be represented
by a dot with horizontal coordinate $d_i$ and vertical coordinate
$F_i$ defined in~(\ref{DiagramCoords}). Then, non-dominated in
sense of Def.~\ref{DefNonDomination} hyperintervals are located on
the lower-right convex hull of the set of dots representing the
hyperintervals.
\end{thm}

We conclude this theoretical consideration by the following
remark. As it has been shown in~\cite{Sergeyev(2000)}, the lengths
of the main diagonals of hyperintervals generated by the new
subdivision strategy from Section~\ref{sectionNewStrategy} are not
arbitrary, contrarily to traditional diagonal schemes (see, e.g.,
\cite{Kvasov:et:al.(2003), Molinaro:et:al.(2001), Pinter(1986),
Pinter(1996)}). They are members of a sequence of values depending
both on the size of the initial hypercube $D=[a,b]$ and on the
number of executed subdivisions. In such a way, the hyperintervals
of a current partition $\{D^k\}$ form several groups. Each group
is characterized by the length of the main diagonals of
hyperintervals within the group. In two-dimensional diagram ($d_i,
F_i$), where $d_i$ and~$F_i$ are from~(\ref{DiagramCoords}), the
hyperintervals from a group are represented by dots with the same
horizontal coordinate $d_i$. For example, in
Fig.~\ref{fig_Optimal} there are seven different groups of
hyperintervals with the horizontal coordinates equal to $d_A$,
$d_B$, $d_C$, $d_E$, $d_F$, $d_G$, and $d_H$. Note that some
groups of a current partition can be empty (see, e.g., the group
with the horizontal coordinate between $d_H$ and $d_G$ in
Fig.~\ref{fig_Optimal}). These groups correspond to diagonals
which are not present in the current partition, but can be created
(or were created) at the successive (previous) iterations of the
algorithm.

It is possible to demonstrate (see~\cite{Sergeyev(2000)}) that
there exists a correspondence between the length of the main
diagonal of a hyperinterval $D_i \in \{D^k\}$ and a non-negative
integer number. This number indicates how many partitions have
been performed starting from the initial hypercube $D$ to obtain
the hyperinterval~$D_i$. At each iteration $k \geq 1$ it can be
considered as an index $l=l(k)$ of the corresponding group of
hyperintervals having the same length of their main diagonals,
where
 \be \label{GroupIndex}
  0 \leq q(k) \leq l(k) \leq Q(k) < +\infty
 \ee
and $q(k)=q$ and $Q(k)=Q$ are indices corresponding to the groups
of the largest and smallest hyperintervals of $\{D^k\}$,
respectively. When the algorithm starts, there exists only one
hyperinterval
--- the admissible region~$D$ --- which belongs to the group with
the index $l=0$. In this case, both indices $q$ and $Q$ are equal
to zero. When a hyperinterval $D_i \in \{D^k\}$ from a group
$l'=l'(k)$ is subdivided, all three generated hyperintervals are
placed into the group with the index $l'+1$. Thus, during the work
of the algorithm, diagonals of hyperintervals become smaller and
smaller, while the corresponding indices of groups of
hyperintervals grow consecutively starting from zero.

For example, in Fig.~\ref{fig_Optimal} there are seven non-empty
groups of hyperintervals of a partition $\{D^k\}$ and one empty
group. The index $q(k)$ (index $Q(k)$) corresponds to the group of
the largest (smallest) hyperintervals represented in
Fig.~\ref{fig_Optimal} by dots with the horizontal coordinate
equal to $d_A$ ($d_H$). For Fig.~\ref{fig_Optimal} we have
$Q(k)=q(k)+7$. The empty group has the index $l(k)=Q(k)-1$.
Suppose that the hyperintervals $D_A$, $D_H$ and~$D_G$
(represented in Fig.~\ref{fig_Optimal} by the dots $A$, $H$, and
$G$, respectively) will be subdivided at the $k$-th iteration. In
this case, the smallest index will remain the same, i.e.,
$q(k+1)=q(k)$, since the group of the largest hyperintervals will
not be empty, while the biggest index will increase, i.e., $Q(k+1)
= Q(k)+1$, since a new group of the smallest hyperintervals will
be created. The previously empty group $Q(k)-1$ will be filled up
by the new hyperintervals generated by partitioning the
hyperinterval $D_G$ and will have the index $l(k+1)=Q(k+1)-2$.

\section{New algorithm}

In this section, a new algorithm for solving problem
(\ref{(1.1)})--(\ref{(1.3)}) is described. First, the new
algorithm is presented and briefly commented. Then its convergence
properties are analyzed.

The new algorithm is oriented on solving difficult
multidimensional multiextremal problems. To accomplish this task,
a two-phase approach consisting of explicitly defined global and
local phases is proposed. It is well known that DIRECT also
balances global and local information during its work. However,
the local phase is too pronounced in this balancing. As has been
already mentioned in Introduction, DIRECT executes too many
function trials in regions of local optima and, therefore,
manifests too slow convergence to the global minimizers when the
objective function has many local minimizers.

In the new algorithm, when a sufficient number of subdivisions of
hyperintervals near the current best point has been performed, the
two-phase approach forces the new algorithm to switch to the
exploration of large hyperintervals that could contain better
solutions. Since many subdivisions have been executed around the
current best point, its neighborhood contains only small
hyperintervals and large ones can be located only far from it.
Thus, the new algorithm balances global and local search in a more
sophisticated way trying to provide a faster convergence to the
global minimizers of difficult multiextremal functions.

Thus, the new algorithm consists of the following two phases:
local improvement of the current best function value (local phase)
and examination of large unexplored hyperintervals in pursuit of
new attraction regions of deeper local minimizers (global phase).
Each of these phases can consist of several iterations. During the
local phase the algorithm tries to explore better the subregion
around the current best point. This phase finishes when the
following two conditions are verified: (i) an improvement on at
least $1\%$ of the minimal function value is not more reached and
(ii) a hyperinterval containing the current best point becomes the
smallest one. After the end of the local phase the algorithm is
switched to the global phase.

The global phase consists of subdividing mainly large
hyperintervals, located possibly far from the current best point.
It is performed until a function value improving the current
minimal value on at least $1\%$ is obtained. When this happens,
the algorithm switches to the local phase during which the
obtained new solution is improved locally. During its work the
algorithm can switch many times from the local phase to the global
one. The algorithm stops when the number of generated trial points
reaches the maximal allowed number.

We assume without loss of generality that the admissible region
$D=[a,b]$ in~(\ref{(1.3)}) is an $N$-dimensional hypercube.
Suppose that at the iteration $k \geq 1$ of the algorithm a
partition $\{D^k\}$ of $D=[a,b]$ has been obtained by the
partitioning procedure
from~(\ref{PartitionD})--(\ref{partition:Delta_m}). Suppose also
that each hyperinterval $D_i \in \{D^k\}$ is represented by a dot
in the two-dimensional diagram $(d_i, F_i)$, where $d_i$ and $F_i$
are from~(\ref{DiagramCoords}), and the groups of hyperintervals
with the same length of their main diagonals are numerated by
indices within a range from $q(k)$ up to $Q(k)$
from~(\ref{GroupIndex}).

To describe the algorithm formally, we need the following
additional designations:

$f_{min}(k)$ -- the best function value (the term `record' will be
also used) found after~$k-1$ iterations.

$x_{min}(k)$ -- coordinates of $f_{min}(k)$.

$D_{min}(k)$ -- the hyperinterval containing the point
$x_{min}(k)$ (if $x_{min}(k)$ is a common vertex of several --- up
to $2^N$ --- hyperintervals, then the smallest hyperinterval is
considered).

$f_{min}^{prec}$ -- the old record. It serves to memorize the
record $f_{min}(k)$ at the start of the current phase (local or
global). The value of $f_{min}^{prec}$ is updated when an
improvement of the current record on at least~$1\%$ is obtained.

$\xi$ -- the parameter of the algorithm, $\xi \geq 0$. It prevents
the algorithm from subdividing already well-explored small
hyperintervals. If $D_t \in \{D^k\}$ is a non-dominated
hyperinterval with respect to an estimate $\tilde{L}$ of the
Lipschitz constant $L$, then this hyperinterval can be subdivided
at the $k$-th iteration only if the following condition is
satisfied
 \be  \label{ConditionOfImprovement}
   R_t(\tilde{L}) \leq f_{min} (k) - \xi,
 \ee
where the lower bound $R_t(\tilde{L})$ is calculated by
formula~(\ref{R_i}). The value of $\xi$ can be set in different
ways (see Section~\ref{sectionResults}).

$T_{max}$ -- the maximal allowed number of trial points that the
algorithm may generate. The algorithm stops when the number of
generated trial points reaches $T_{max}$. During the course of the
algorithm the satisfaction of this termination criterion is
verified after every subdivision of a hyperinterval.

$Lcounter$, $Gcounter$ -- the counters of iterations executed
during the current local and global phases, respectively.

$p(k)$ -- the index of the group the hyperinterval $D_{min}(k)$
belongs to. Notice that the inequality $q(k) \leq p(k) \leq Q(k)$
is satisfied for any iteration number $k$. Since both local and
global phases can embrace more than one iteration, the index
$p(k)$ (as well as the indices $q(k)$ and $Q(k)$) can change
(namely, increase) during these phases. Note also that the
group~$p(k)$ can be different from the groups containing
hyperintervals with the smallest sum of the objective function
values at their vertices (see two groups of hyperintervals
represented in Fig.~\ref{fig_Optimal} by the horizontal
coordinates equal to $d_G$ and~$d_F$). Moreover, the hyperinterval
$D_{min}(k)$ is not represented necessarily by the `lowest' point
from the group $p(k)$ in the two-dimensional diagram $(d_i, F_i)$
-- even if the current best function value is obtained at a vertex
of $D_{min}(k)$, the function value at the other vertex can be too
high and the sum of these two values can be greater than the
corresponding value of another hyperinterval from the group
$p(k)$.

$p'$ -- the index of the group containing the hyperinterval
$D_{min}(k)$ at the start of the current phase (local or global).
Hyperintervals from the groups with indices greater than $p'$ are
not considered when non-dominated hyperintervals are looked for.
Whereas the index $p(k)$ can assume different values during the
current phase, the index $p'$ remains, as a rule, invariable. It
is changed only when it violates the left part of
condition~(\ref{GroupIndex}). This can happen when groups with the
largest hyperintervals disappear and, therefore, the index $q(k)$
increases and becomes equal to $p'$. In this case, the index $p'$
increases jointly with~$q(k)$.

$p''$ -- the index of the group immediately preceding the group
$p'$, i.e., $p''=p'-1$. This index is used within the local phase
and can increase if $q(k)$ increases during this phase.

$r'$ -- the index of the middle group of hyperintervals between
the groups $p'$ and~$q(k)$, i.e., $r'=\lceil(q(k)+p')/2\rceil$.
This index is used within the global phase as a separator between
the groups of large and small hyperintervals. It can increase
if~$q(k)$ increases during this phase.

To clarify the introduced group indices, let us consider an
example of a partition~$\{D^k\}$ represented by the
two-dimensional diagram in Fig.~\ref{fig_Optimal}. Let us suppose
that the index $q(k)$ of the group of the largest hyperintervals
corresponding to the points with the horizontal coordinate $d_A$
in Fig.~\ref{fig_Optimal} is equal to 10. The index $Q(k)$ of the
group of the smallest hyperintervals with the main diagonals equal
to $d_H$ (see Fig.~\ref{fig_Optimal}) is equal to $Q(k)=q(k)+7 =
17$. Let us also assume that the hyperinterval~$D_{min}(k)$
belongs to the group of hyperintervals with the main diagonals
equal to~$d_G$ (see Fig.~\ref{fig_Optimal}). In this case, the
index $p(k)$ is equal to 15 and the index $p'$ is equal to 15 too.
The index $p''=15-1 = 14$ and it corresponds to the group of
hyperintervals represented in Fig.~\ref{fig_Optimal} by the dots
with the horizontal coordinate $d_F$. Finally, the index
$r'=\lceil(10+15)/2\rceil = 13$ and it corresponds to
hyperintervals with the main diagonals equal to $d_E$. The indices
$p'$, $p''$, and $r'$ can change only if the index $q(k)$
increases. Otherwise, they remain invariable during the iterations
of the current phase (local or global). At the same time, the
index $p(k)$ can change at every iteration, as soon as a new best
function value belonging to a hyperinterval of a group different
from $p(k)$ is obtained.

\vspace{2mm} Now we are ready to present a formal scheme of the
new algorithm. \vspace{2mm}

\begin{description}
  \item[{\bf Step 1: Initialization.}] Set the current iteration number $k:=1$,
 the current record $f_{min}(k):=\min\{f(a),f(b)\}$ where $a$ and $b$ are from~(\ref{(1.3)}).
 Set group indices $q(k):=Q(k):=p(k):=0$.
 \item[{\bf Step 2: Local Phase.}] Memorize the current record
 $f_{min}^{prec}:=f_{min}(k)$ and  perform the following steps:

 \begin{description}

  \item[{\bf Step 2.1.}] Set $Lcounter := 1$ and fix the group index $p':=p(k)$.

  \item[{\bf Step 2.2.}] Set $p'' := \max\{p'-1, q(k)\}$.

  \item[{\bf Step 2.3.}] Determine non-dominated
  hyperintervals considering only groups of hyperintervals with the indices
  from $q(k)$ up to $p''$. Subdivide those non-dominated hyperintervals which satisfy
  inequality~(\ref{ConditionOfImprovement}). Set $k:=k+1$.

  \item[{\bf Step 2.4.}] Set $Lcounter := Lcounter + 1$ and check whether
  $Lcounter \leq N$. If this is the case, then go to Step 2.2. Otherwise, go to
  Step 2.5.

  \item[{\bf Step 2.5.}] Set $p'=\max\{p', q(k)\}$. Determine non-dominated
  hyperintervals considering only groups of hyperintervals with the indices
  from $q(k)$ up to $p'$. Subdivide those non-dominated hyperintervals which satisfy
  inequality~(\ref{ConditionOfImprovement}). Set $k:=k+1$.

 \end{description}

 \item[{\bf Step 3: Switch.}] If condition
   \be \label{ConditionOfLocalImprovement}
     f_{min}(k) \leq f_{min}^{prec}- 0.01 | f_{min}^{prec} |
   \ee
 is satisfied, then go to Step 2 and repeat
 the local phase with the new obtained value of the record $f_{min}(k)$.
 Otherwise, if the hyperinterval $D_{min}(k)$ is not the smallest one,
 or the current partition of $D$ consists only of hyperintervals with equal
 diagonals (i.e., $p(k) < Q(k)$
 or $q(k)=Q(k)$), then go to Step 2.1 and repeat
 the local phase with the old record $f_{min}^{prec}$.

 If the obtained improvement of the best function value is not sufficient to
 satisfy~(\ref{ConditionOfLocalImprovement}) and $D_{min}(k)$
 is the smallest hyperinterval of the current partition (i.e., all the following
 inequalities -- (\ref{ConditionOfLocalImprovement}), $p(k) < Q(k)$, and
 $q(k) = Q(k)$ --- fail), then go to Step 4 and perform the global phase.

 \item[{\bf Step 4: Global Phase.}] Memorize the current record
 $f_{min}^{prec}:=f_{min}(k)$ and  perform the following steps:

 \begin{description}

  \item[{\bf Step 4.1.}] Set $Gcounter := 1$ and fix the group index $p':=p(k)$.

  \item[{\bf Step 4.2.}] Set $p'=\max\{p', q(k)\}$ and calculate
  the `middle' group index $r'=\lceil(q(k)+p')/2\rceil$.

  \item[{\bf Step 4.3.}] Determine non-dominated
  hyperintervals considering only groups of hyperintervals with the indices
  from $q(k)$ up to $r'$. Subdivide those non-dominated hyperintervals which satisfy
  inequality~(\ref{ConditionOfImprovement}). Set $k:=k+1$.

  \item[{\bf Step 4.4.}] If
  condition~(\ref{ConditionOfLocalImprovement}) is satisfied, then
  go to Step 2 and perform the local phase with the new obtained value
  of the record $f_{min}(k)$. Otherwise, go to Step 4.5.

  \item[{\bf Step 4.5.}] Set $Gcounter := Gcounter +1$; check
  whether $Gcounter \leq 2^{N+1}$. If this is the case, then go to Step 4.2.
  Otherwise, go to Step 4.6.

  \item[{\bf Step 4.6.}] Set $p'=\max\{p', q(k)\}$. Determine non-dominated
  hyperintervals considering only groups of hyperintervals with the indices
  from $q(k)$ up to $p'$. Subdivide those non-dominated hyperintervals which satisfy
  inequality~(\ref{ConditionOfImprovement}). Set $k:=k+1$.

  \item[{\bf Step 4.7.}] If
  condition~(\ref{ConditionOfLocalImprovement}) is satisfied, then
  go to Step 2 and perform the local phase with the new obtained value
  of the record $f_{min}(k)$.
  Otherwise, go to Step 4.1: update the value of the group index $p'$
  and repeat the global phase with the old record $f_{min}^{prec}$.
 \end{description}

 \end{description}

\vspace{2mm}

Let us give a few comments on the introduced algorithm. It starts
from the local phase. In the course of this phase, it subdivides
non-dominated hyperintervals with the main diagonals greater than
the main diagonal of $D_{min}(k)$ (i.e., from the groups with the
indices from $q(k)$ up to $p'$; see Steps 2.1 -- 2.4). This
operation is repeated $N$ times, where $N$ is the problem
dimension from~(\ref{(1.3)}). Remind that during each subdivision
of a hyperinterval by the
scheme~(\ref{PartitionD})~--~(\ref{partition:Delta_m}) only one
side of the hyperinterval (namely, the longest side given by
formula~(\ref{side_i})) is partitioned. Thus, performing $N$
iterations of the local phase eventually subdivides all $N$ sides
of hyperintervals around the current best point. At the last,
$(N+1)$-th, iteration of the local phase (see Step~2.5)
hyperintervals with the main diagonal equal to $D_{min}(k)$ are
considered too. In such a way, the hyperinterval containing the
current best point can be partitioned too.

Thus, either the current record is improved, or the hyperinterval
providing this record becomes smaller. If the conditions of
switching to the global phase (see Step 3) are not satisfied, the
local phase is repeated. Otherwise, the algorithm switches to the
global phase, avoiding unnecessary evaluations of $f(x)$ within
already well explored subregions.

During the global phase the algorithm searches for better new
minimizers. It performs series of loops (see Steps 4.1 -- 4.7)
while a non-trivial improvement of the best function value is not
obtained, i.e., condition~(\ref{ConditionOfLocalImprovement}) is
not satisfied. Within a loop of the global phase the algorithm
performs a substantial number of subdivisions of large
hyperintervals located far from the current best point, namely,
hyperintervals from the groups with the indices from $q(k)$ up to
$r'$ (see Steps 4.2 -- 4.5). Since each trial point can belong up
to $2^{N}$ hyperintervals, the number of subdivisions should not
be smaller than~$2^N$. The value of this number equal to $2^{N+1}$
has been chosen because it provided a good performance of the
algorithm in our numerical experiments.

Note that the situation when the current best function value is
improved but the amount of this improvement is not sufficient to
satisfy~(\ref{ConditionOfLocalImprovement}), can be verified at
the end of a loop of the global phase (see Step 4.7). In this
case, the algorithm is not switched to the local phase. It
proceeds with the next loop of the global phase, eventually
updating the index $p'$ (see Step 4.1) but non updating the old
record $f_{min}^{prec}$.

Let us now study convergence properties of the new algorithm
during minimization of the function $f(x)$
from~(\ref{(1.1)})--(\ref{(1.3)}) when the maximal allowed number
of generated trial points $T_{max}$ is equal to infinity. In this
case, the algorithm does not stop (the number of iterations $k$
goes to infinity) and an infinite sequence of trial
points~$\{x^{j(k)}\}$ is generated. The following theorem
establishes the so-called `everywhere dense' convergence of the
new algorithm.

\begin{thm} \label{TheoremConvergence}
For any point $x \in D$ and any $\delta >0$ there exist an
iteration number $k(\delta) \geq 1$ and a point $x' \in
\{x^{j(k)}\}$, $k > k(\delta)$, such that $\|x - x'\| < \delta$.
\end{thm}

\begin{proof}
Trial points generated by the new algorithm are vertices of the
main diagonals of hyperintervals. Due
to~(\ref{PartitionD})--(\ref{partition:Delta_m}), every
subdivision of a hyperinterval produces three new hyperintervals
with the volume equal to a third of the volume of the subdivided
hyperinterval and the proportionally smaller main diagonals. Thus,
fixed a positive value of $\delta$, it is sufficient to prove that
after a finite number of iterations~$k(\delta)$ the largest
hyperinterval of the current partition of $D$ will have the main
diagonal smaller than $\delta$. In such a case, in
$\delta$-neighborhood of any point of $D$ there will exist at
least one trial point generated by the algorithm.

To see this, let us fix an iteration number $k'$ and consider the
group $q(k')$ of the largest hyperintervals of a partition
$\{D^{k'}\}$. As it can be seen from the scheme of the algorithm,
for any $k' \geq 1$ this group is taken into account when
non-dominated hyperintervals are looked for. Moreover, a
hyperinterval $D_t \in \{D^{k'}\}$ from this group having the
smallest sum of the objective function values at its vertices is
partitioned at each iteration $k \geq 1$ of the algorithm. This
happens because there always exists a sufficiently large estimate
$L_\infty$ of the Lipschitz constant $L$ such that the
hyperinterval~$D_t$ is a non-dominated hyperinterval with respect
to $L_\infty$ and condition~(\ref{ConditionOfImprovement}) is
satisfied for the lower bound $R_t(L_\infty)$ (see
Fig.~\ref{fig_Optimal}). Three new hyperintervals generated during
the subdivision of $D_t$ by using the
strategy~(\ref{PartitionD})--(\ref{partition:Delta_m}) are
inserted into the group with the index $q(k')+1$. Hyperintervals
of the group $q(k')+1$ have the volume equal to a third of the
volume of hyperintervals of the group $q(k')$.

Since each group contains only finite number of hyperintervals,
after a sufficiently large number of iterations $k
> k'$ all hyperintervals of the group $q(k')$ will be subdivided.
The group $q(k')$ will become empty and the index of the group of
the largest hyperintervals will increase, i.e., $q(k) = q(k')+1$.
Such a procedure will be repeated with a new group of the largest
hyperintervals. So, when the number of iterations grows, the index
$q(k)$ increases and due to (\ref{GroupIndex}) the index $Q(k)$
increases too. This means, that there exists a finite number of
iterations $k(\delta)$ such that after performing~$k(\delta)$
iterations of the algorithm the largest hyperinterval of the
current partition $\{D^{k(\delta)}\}$ will have the main diagonal
smaller than $\delta$. \qquad
\end{proof}

\section{Numerical results} \label{sectionResults}

In this section, we present results performed to compare the new
algorithm with two methods belonging to the same class: the
original DIRECT algorithm from~\cite{Jones:et:al.(1993)} and its
locally-biased modification DIRECT{\it l} \hspace{1mm}from
\cite{Gablonsky(2001), Gablonsky&Kelley(2001)}. The implementation
of these two methods described in~\cite{Gablonsky(1998),
Gablonsky(2001)} and downloadable from~\cite{Kelley:HomePage} has
been used in all the experiments.

To execute a numerical comparison, we need to define the parameter
$\xi$ of the algorithm from~(\ref{ConditionOfImprovement}). This
parameter can be set either independently from the current record
$f_{min} (k)$, or in a relation with it. Since the objective
function $f(x)$ is supposed to be ``black-box'', it is not
possible to know a priori which of these two ways is better.

In DIRECT~\cite{Jones:et:al.(1993)}, where a similar parameter is
used, a value $\xi$ related to the current minimal function value
$f_{min} (k)$ is fixed as follows
 \be  \label{DirEpsilon}
    \xi = \varepsilon | f_{min} (k) |, \hspace{3mm} \varepsilon
    \geq 0.
 \ee

The choice of $\varepsilon$ between $10^{-3}$ and $10^{-7}$ has
demonstrated good results for DIRECT on a set of test functions
(see \cite{Jones:et:al.(1993)}). Later formula~(\ref{DirEpsilon})
has been used by many authors (see,
e.g.,~\cite{Carter:et:al.(2001), Finkel&Kelley(2004),
Gablonsky(2001), Gablonsky&Kelley(2001), He:et:al.(2002)}) and
also has been realized in the implementation of DIRECT
(see~\cite{Gablonsky(1998), Gablonsky(2001)}) taken for numerical
comparison with the new algorithm. Since the value of $\varepsilon
= 10^{-4}$ recommended in~\cite{Jones:et:al.(1993)} has produced
the most robust results for~DIRECT (see, e.g.,
\cite{Gablonsky(2001), Gablonsky&Kelley(2001), He:et:al.(2002),
Jones:et:al.(1993)}), exactly this value was used
in~(\ref{DirEpsilon}) for DIRECT in our numerical experiments. In
order to have comparable results, the same
formula~(\ref{DirEpsilon}) and $\varepsilon = 10^{-4}$ were used
in the new algorithm too.

The global minimizer $x^* \in D$ was considered to be found when
an algorithm generated a trial point~$x'$ inside a hyperinterval
with a vertex $x^*$ and the volume smaller than the volume of the
initial hyperinterval $D=[a,b]$ multiplied by an accuracy
coefficient~$\Delta$, $0< \Delta \leq 1$, i.e.,
 \be \label{x*Found}
  |x'(i) - x^*(i) | \leq \sqrt[N]{\Delta}(b(i)-a(i)), \hspace{3mm}
  1 \leq i \leq N,
 \ee
where $N$ is from~(\ref{(1.3)}). This condition means that, given
$\Delta$, a point $x'$ satisfies~(\ref{x*Found}) if the
hyperinterval with the main diagonal $[x', x^*]$ and the sides
proportional to the sides of the initial hyperinterval $D=[a,b]$
has a volume at least $\Delta^{-1}$ times smaller than the volume
of $D$. Note that if in~(\ref{x*Found}) the value of $\Delta$ is
fixed and the problem dimension~$N$ increases, the length of the
diagonal of the hyperinterval $[x', x^*]$ increases too. In order
to avoid this undesirable growth, the value of $\Delta$ was
progressively decreased when the problem dimension increased.

We stopped the algorithm either when the maximal number of trials
$T_{max}$ was reached, or when condition~(\ref{x*Found}) was
satisfied. Note that such a type of stopping criterion is
acceptable only when the global minimizer $x^*$ is known, i.e.,
for the case of test functions. When a real ``black-box''
objective function is minimized and global minimization algorithms
have an internal stopping criterion, they execute a number of
iterations (that can be very high) after a `good' estimate
of~$f^*$ has been obtained in order to demonstrate a `goodness' of
the found solution (see, e.g., \cite{Horst&Pardalos(1995),
Pinter(1996), Strongin&Sergeyev(2000)}).

\begin{table}[t]
\caption{Number of trial points for test functions used
in~\cite{Jones:et:al.(1993)}.}
\begin{center} \footnotesize \label{table1}
\begin{tabular}{|c|c|c|c|c|c|c|}\hline
Function & $N$ & $D=[a,b]$ & $\Delta$ &DIRECT &DIRECT{\it l} &New \\
\hline
Shekel 5 & 4 & $[0,10]^4$ & $10^{-6}$ &57 &53 &208 \\
Shekel 7 & 4 & $[0,10]^4$ & $10^{-6}$ &53 &45 &1465 \\
Shekel 10 & 4 & $[0,10]^4$ & $10^{-6}$ &53 &45 &1449 \\
Hartman 3 & 3 & $[0,1]^3$ & $10^{-6}$ &113 &79 & 137\\
Hartman 6 & 6 & $[0,1]^6$ & $10^{-7}$ &144 &78 & 4169 \\
Branin RCOS & 2 &$[-5,10]\times[0,15]$ & $10^{-4}$ &41 &31 &76 \\
Goldstein and Price & 2 &$[-2,2]^2$& $10^{-4}$ &37 &29 & 99 \\
Six-Hump Camel & 2 &$[-3,3]\times[-2,2]$& $10^{-4}$ &105 &127 &128 \\
Shubert & 2 &$[-8,10]^2$& $10^{-4}$ &19 &15 & 59\\ \hline
\end{tabular}
\end{center}
\end{table}

In the first series of experiments, test functions
from~\cite{Dixon&Szego(1978)} and~\cite{Yao(1989)} were used
because in~\cite{Gablonsky(2001), Gablonsky&Kelley(2001),
Jones:et:al.(1993)} DIRECT and DIRECT{\it l} have been tested on
these functions. It can be seen from Table~\ref{table1} that both
methods DIRECT and DIRECT{\it l} have executed a very small amount
of trials until they generated a point in a
neighborhood~(\ref{x*Found}) of a global minimizer. For example,
condition~(\ref{x*Found}) was satisfied for the six-dimensional
Hartman's function after 78 (144) trials performed by DIRECT{\it
l} (DIRECT). Such a small number of trials is explained by a
simple structure of the function. We observe, in accordance with
\cite{Torn:et:al.(1999)}, that the test functions
from~\cite{Dixon&Szego(1978)} used in~\cite{Jones:et:al.(1993)}
are not suitable for testing global optimization methods. These
functions are characterized by a small chance to miss the region
of attraction of the global minimizer (see
\cite{Torn:et:al.(1999)}). Usually, when a real difficult
``black-box'' function of high dimension is minimized, the number
of trials that it is necessary to execute to place a trial point
in the neighborhood of the global minimizer is significantly
higher. The algorithm proposed in this paper is oriented on such a
type of functions. It tries to perform a good examination of the
admissible region in order to reduce the risk of missing the
global solution. Therefore, for simple test functions of
Table~\ref{table1} and the stopping rule~(\ref{x*Found}) it
generated more trial points than DIRECT or DIRECT{\it l}.

Hence, more sophisticated test problems are required for carrying
out numerical comparison among global optimization algorithms (see
also the related discussion in~\cite{Khompatraporn:et:al.(2005)}).

Many difficult global optimization tests can be taken from
real-life applications (see, e.g., \cite{Floudas:et:al(1999)} and
bibliographic references within it). But the lack of comprehensive
information (such as number of local minima, their locations,
attraction regions, local and global values, ecc.) describing
these tests creates an obstacle in verifying efficiency of the
algorithms. Very frequently it is also difficult to fix properly
many correlated parameters determining some test functions because
often the sense of these parameters is not intuitive, especially
in high dimensions. Moreover, tests may differ too much one from
another and as a result it is not possible to have many test
functions with similar properties. Therefore, the use of
randomly-generated classes of test functions having similar
properties can be a reasonable solution for a satisfactory
comparison.

Thus, in our numerical experiments we used the GKLS-generator
described in~\cite{Gaviano:et:al.(2003)} (and free-downloadable
from {\it http://wwwinfo.deis.unical.it/ $\tilde{
}$yaro/GKLS.html}). It generates classes of multidimensional and
multiextremal test functions with known local and global minima.
The procedure of generation consists of defining a convex
quadratic function (paraboloid) systematically distorted by
polynomials. Each test class provided by the generator includes
100 functions and is defined only by the following five
parameters:

$N$ -- problem dimension;

$M$ -- number of local minima;

$f^*$ -- value of the global minimum;

$\rho^*$ -- radius of the attraction region of the global
minimizer;

$r^*$ -- distance from the global minimizer to the vertex of the
paraboloid.

The other necessary parameters are chosen randomly by the
generator for each test function of the class. Note, that the
generator produces always the same test classes for a given set of
the user-defined parameters allowing one to perform repeatable
numerical experiments.

\begin{figure}[!t] \vspace{2mm}
\centerline{\epsfig{file=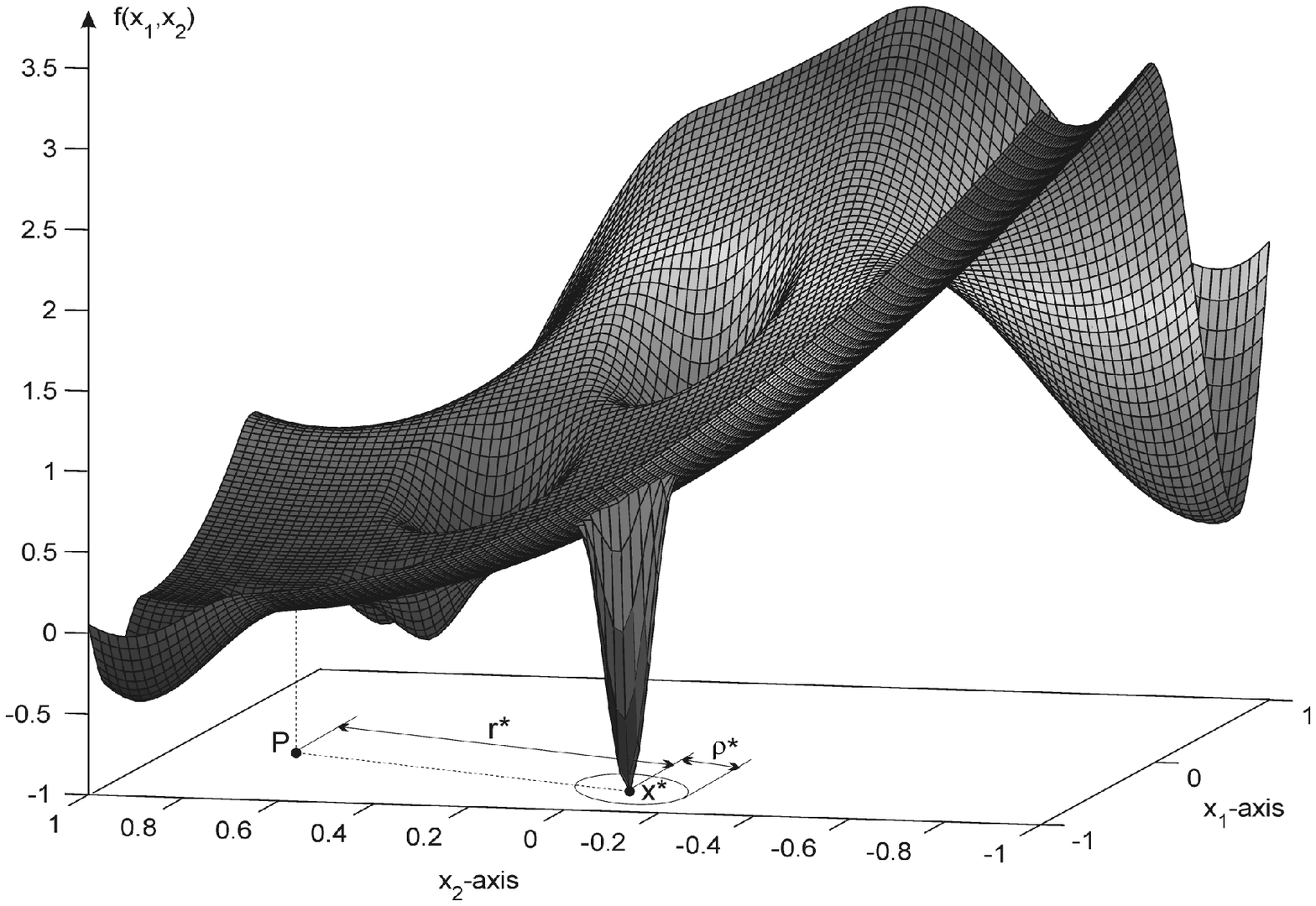,width=115mm,height=90mm,angle=0,silent=}}
\vspace{-2 mm}\caption{An example of the two-dimensional function
from the GKLS test class.} \label{fig_GKLS}
\end{figure}

By changing the user-defined parameters, classes with different
properties can be created. For example, fixed dimension of the
functions and number of local minima, a more difficult class can
be created either by shrinking the attraction region of the global
minimizer, or by moving the global minimizer far away from the
paraboloid vertex.

For conducting numerical experiments, we used eight GKLS classes
of continuously differentiable test functions of dimensions $N=2$,
3, 4, and 5. The number of local minima $M$ was equal to 10 and
the global minimum value $f^*$ was equal to $-1.0$ for all classes
(these values are default settings of the generator). For each
particular problem dimension $N$ we considered two test classes: a
simple class and a difficult one. The difficulty of a class was
increased either by decreasing the radius $\rho^*$ of the
attraction region of the global minimizer (as for two- and
five-dimensional classes), or by increasing the distance $r^*$
from the global minimizer $x^*$ to the paraboloid vertex~$P$
(three- and four-dimensional classes).

In Fig.~\ref{fig_GKLS}, an example of a test function from the
following continuously differentiable GKLS class is given: $N=2$,
$M=10$, $f^*=-1$, $\rho^*=0.10$, and $r^*=0.90$. This function is
defined over the region $D=[-1,1]^2$ and its number is~87 in the
given test class. The randomly generated global minimizer of this
function is $x^*=(-0.767, -0.076)$ and the coordinates of~$P$ are
$(-0.489, 0.780)$. Results for the whole class to which the
function from Fig.~\ref{fig_GKLS} belongs to are given in
Tables~\ref{table2} and~\ref{table3} on the second line.

We stopped algorithms either when the maximal number of
trials~$T_{max}$ equal to~1\,000\,000 was reached, or when
condition~(\ref{x*Found}) was satisfied. To describe experiments,
we introduce the following designations:

$T_s$ -- the number of trials performed by the method under
consideration to solve the problem number $s$, $1 \leq s \leq
100$, of a fixed test class. If the method was not able to solve a
problem $j$ in less than $T_{max}$ function evaluations, $T_j$
equal to $T_{max}$ was taken.

$m_s$ -- the number of hyperintervals generated to solve the
problem $s$.

The following four criteria were used to compare the methods.

\vspace{1mm}

{\bf Criterion C1.} Number of trials $T_{s^*}$ required for a
method to satisfy condition~(\ref{x*Found}) for {\it all} 100
functions of a particular test class, i.e.,
 \begin{equation}
   T_{s^*} = \max_{1 \leq s \leq 100} T_s, \hspace{5mm}
   s^* = \arg\max_{1 \leq s \leq 100} T_s. \label{f_num}
 \end{equation}

{\bf Criterion C2.} The corresponding number of hyperintervals,
$m_{s^*}$, generated by the method, where $s^*$ is
from~(\ref{f_num}).

{\bf Criterion C3.} Average number of trials $T_{avg}$ performed
by the method during minimization of {\it all} 100 functions from
a particular test class, i.e.,
 \begin{equation}
  T_{avg} = \frac{1}{100}\sum_{s=1}^{100} T_s. \label{avg_trials}
 \end{equation}

{\bf Criterion C4.} Number $p$ (number $q$) of functions from a
class for which DIRECT or DIRECT{\it l} executed less (more)
function evaluations than the new algorithm. If~$T_s$ is the
number of trials performed by the new algorithm and $T_s'$ is the
corresponding number of trials performed by a competing method,
$p$ and $q$ are evaluated as follows
 \begin{equation}
  p = \sum_{s=1}^{100} \delta_s', \hspace{5mm} \label{winDIRECT}
  \delta_s' =
  \left\{
    \begin{array}{ll}
       1, & T_s' < T_s,  \\
       0, & {\rm otherwise}.
   \end{array} \right.
 \end{equation}
 \begin{equation}
  q= \sum_{s=1}^{100} \delta_s, \hspace{5mm} \label{winNEW}
  \delta_s =
  \left\{
    \begin{array}{ll}
       1, & T_s < T_s',  \\
       0, & {\rm otherwise}.
   \end{array} \right.
 \end{equation}
If $p+q < 100$ then both the methods under consideration solve the
remaining ($100-p-q$) problems with the same number of function
evaluations.

\vspace{1mm}

Note that results based on Criteria C1 and C2 are mainly
influenced by minimization of the most difficult functions of a
class. Criteria C3 and C4 deal with average data of a class.

Criterion C1 is of the fundamental importance for the methods
comparison on the whole test class because it shows how many
trials it is necessary to execute to solve {\it all} the problems
of a class. Thus, it represents the worst case results of the
given method on the fixed class.

At the same time, the number of generated hyperintervals
(Criterion C2) provides an important characteristic of any
partition algorithm for solving~(\ref{(1.1)})--(\ref{(1.3)}). It
reflects indirectly the degree of qualitative examination of $D$
during the search for a global minimum. The greater is this
number, the more information about the admissible domain is
available and, therefore, the smaller should be the risk to miss
the global minimizer. However, algorithms should not generate many
redundant hyperintervals since this slows down the search and is
therefore a disadvantage of the method.

Let us first compare the three methods on Criteria C1 and C2.
Results of numerical experiments with eight GKLS tests classes are
shown in Tables \ref{table2} and \ref{table3}. The accuracy
coefficient $\Delta$ from~(\ref{x*Found}) is given in the second
column of the tables. Table~\ref{table2} reports the maximal
number of trials required for satisfying condition~(\ref{x*Found})
for a half of the functions of a particular class (columns
``50\%'') and for all 100 function of the class (columns
``100\%''). The notation `$>$ 1\,000\,000 $(j)$' means that after
1\,000\,000 function evaluations the method under consideration
was not able to solve~$j$ problems. The corresponding numbers of
generated hyperintervals are indicated in Table~\ref{table3}.
Since DIRECT and DIRECT{\it l} use during their work the
central-sampling partition strategy, the number of generated trial
points and the number of generated hyperintervals coincide for
these methods.

Note that on a half of test functions from each class (which were
the most simple for each method with respect to the other
functions of the class) the new algorithm manifested a good
performance with respect to DIRECT and DIRECT{\it l} in terms of
the number of generated trial points (see columns ``50\%'' in
Table~\ref{table2}). When all the functions were taken in
consideration (and, consequently, difficult functions of the class
were considered too), the number of trials produced by the new
algorithm was much fewer in comparison with two other methods (see
columns ``100\%'' in Table~\ref{table2}), ensuring at the same
time a substantial examination of the admissible domain (see
Table~\ref{table3}).

\begin{table}[t]
\caption{Number of trial points for GKLS test functions (Criterion
C1).}
\begin{center} \scriptsize \label{table2}
\begin{tabular}{@{\extracolsep{\fill}}|c|c|c|c|r|r|r|r|r|r|}\hline
$N$ & $\Delta$ & \multicolumn{2}{c|}{Class} &
\multicolumn{3}{c|}{50\%} & \multicolumn{3}{c|}{100\%}\\
\cline{3-10}& &$r^*$ &$\rho^*$ &DIRECT &DIRECT{\it l} &New &DIRECT &DIRECT{\it l} &New\\
\hline
2 &$10^{-4}$ &.90 &.20 &111  &152 &166  &1159  &2318 &403\\
2 &$10^{-4}$ &.90 &.10 &1062  &1328 &613  &3201  &3414 &1809\\
\hline
3 &$10^{-6}$ &.66 &.20 &386   &591 &615  &12507 &13309 &2506\\
3 &$10^{-6}$ &.90 &.20 &1749  &1967 &1743 &$>$1000000 (4) &29233 & 6006\\
\hline
4 &$10^{-6}$ &.66 &.20 &4805  &7194 &4098  &$>$1000000 (4) &118744 &14520\\
4 &$10^{-6}$ &.90 &.20 &16114 &33147 &15064 &$>$1000000 (7)
&287857 &42649\\ \hline
5 &$10^{-7}$ &.66 &.30 &1660  &9246 &3854  &$>$1000000 (1) &178217 &33533\\
5 &$10^{-7}$ &.66 &.20 &55092  &126304 &24616  &$>$1000000 (16) &$>$1000000 (4) &93745\\
\hline
\end{tabular}
\end{center}
\end{table}

\begin{table}[t]
\caption{Number of hyperintervals for GKLS test functions
(Criterion C2).}
\begin{center} \scriptsize \label{table3}
\begin{tabular}{@{\extracolsep{\fill}}|c|c|c|c|r|r|r|r|r|r|}\hline
$N$ & $\Delta$ & \multicolumn{2}{c|}{Class} &
\multicolumn{3}{c|}{50\%} & \multicolumn{3}{c|}{100\%} \\
\cline{3-10}& &$r^*$ &$\rho^*$ &DIRECT &DIRECT{\it l} &New &DIRECT &DIRECT{\it l} &New\\
\hline
2 &$10^{-4}$ &.90 &.20 &111  &152 &269  &1159  &2318 & 685\\
2 &$10^{-4}$ &.90 &.10 &1062 &1328 &1075 &3201 &3414 &3307 \\
\hline
3 &$10^{-6}$ &.66 &.20 &386  &591  &1545 &12507 &13309 &6815\\
3 &$10^{-6}$ &.90 &.20 &1749 &1967 &5005 &$>$1000000 &29233 &17555\\
\hline
4 &$10^{-6}$ &.66 &.20 &4805 &7194 &15145 &$>$1000000 &118744 &73037\\
4 &$10^{-6}$ &.90 &.20 &16114 &33147 &68111 &$>$1000000&287857 &211973\\
\hline
5 &$10^{-7}$ &.66 &.30 &1660 &9246 &21377 &$>$1000000&178217 &206323\\
5 &$10^{-7}$ &.66 &.20 &55092  &126304 &177927  &$>$1000000 &$>$1000000 &735945\\
\hline
\end{tabular}
\end{center}
\end{table}

In our opinion, the impossibility of DIRECT to determine global
minimizers of several test functions is related to the following
fact. DIRECT found quickly the vertex of the paraboloid (at which
the function value is set by default equal to 0) used for
determining GKLS test functions. Hence, the parameter $\xi$ was
very close to zero (due to~(\ref{DirEpsilon})) and condition
similar to~(\ref{ConditionOfImprovement}) was satisfied for almost
all small hyperintervals. Moreover, many small hyperintervals
around the paraboloid vertex with function values close one to
another and to the current minimal value were created. In such a
situation, DIRECT subdivided many of these hyperintervals. Thus,
at each iteration DIRECT partitioned a large amount of small
hyperintervals and, therefore, was not able to go out from the
attraction region of the paraboloid vertex.

\begin{table}[t]
\caption{Number of trial points for shifted GKLS test functions
(Criterion C1).}
\begin{center}\scriptsize \label{table4}
\begin{tabular}{@{\extracolsep{\fill}}|c|c|c|c|r|r|r|r|r|r|}\hline
$N$ & $\Delta$ & \multicolumn{2}{c|}{Class} &
\multicolumn{3}{c|}{50\%} & \multicolumn{3}{c|}{100\%}\\
\cline{3-10}& &$r^*$ &$\rho^*$ &DIRECT &DIRECT{\it l} &New &DIRECT &DIRECT{\it l} &New\\
\hline
2 &$10^{-4}$ &.90 &.20 &111  &146 &165  &1087  &1567 &403\\
2 &$10^{-4}$ &.90 &.10 &911   &1140 &508  &2973  &2547 &1767 \\
\hline
3 &$10^{-6}$ &.66 &.20 &364   &458 &606  &6292&10202 &1912 \\
3 &$10^{-6}$ &.90 &.20 &1485   &1268 &1515  &14807 &28759 &4190 \\
\hline
4 &$10^{-6}$ &.66 &.20 &4193   &4197 &3462  &37036 &95887 &14514 \\
4 &$10^{-6}$ &.90 &.20 &14042   &24948 &11357  &251801 &281013 &32822 \\
\hline
5 &$10^{-7}$ &.66 &.30 &1568   &3818 &3011  &102869 &170709 &15343 \\
5 &$10^{-7}$ &.66 &.20 &32926   &116025 &15071  &454925 &$>1000000 (1)$ &77981\\
\hline
\end{tabular}
\end{center}
\end{table}

\begin{table}[!t]
\caption{Number of hyperintervals for shifted GKLS test functions
(Criterion C2).}
\begin{center}\scriptsize \label{table5}
\begin{tabular}{@{\extracolsep{\fill}}|c|c|c|c|r|r|r|r|r|r|}\hline
$N$ & $\Delta$ & \multicolumn{2}{c|}{Class} &
\multicolumn{3}{c|}{50\%} & \multicolumn{3}{c|}{100\%} \\
\cline{3-10}& &$r^*$ &$\rho^*$ &DIRECT &DIRECT{\it l} &New &DIRECT &DIRECT{\it l} &New\\
\hline
2 &$10^{-4}$ &.90 &.20 &111  &146 &281  &1087  &1567 &685\\
2 &$10^{-4}$ &.90 &.10 &911   &1140 &905  &2973  &2547 &3227\\
\hline
3 &$10^{-6}$ &.66 &.20 &364   &458 &1585  &6292 &10202 &5337 \\
3 &$10^{-6}$ &.90 &.20 &1485   &1268 &4431  &14807 &28759 &12949 \\
\hline
4 &$10^{-6}$ &.66 &.20 &4193   &4197 &14961  &37036 &95887 &73049\\
4 &$10^{-6}$ &.90 &.20 &14042   &24948 &57111  &251801 &281013 &181631\\
\hline
5 &$10^{-7}$ &.66 &.30 &1568   &3818 &17541  &102869 &170709 &106359 \\
5 &$10^{-7}$ &.66 &.20 &32926   &116025 &108939  &454925 &$>1000000$ &685173\\
\hline
\end{tabular}
\end{center}
\end{table}

\begin{table}[!t]
\caption{Improvement obtained by the new algorithm in terms of
Criterion C1.}
\begin{center}\scriptsize \label{table6}
\begin{tabular}{@{\extracolsep{\fill}}|c|c|c|c|c|c|c|c|}\hline
$N$ & $\Delta$ & \multicolumn{2}{c|}{Class} &
\multicolumn{2}{c|}{GKLS} & \multicolumn{2}{c|}{Shifted GKLS} \\
\cline{3-8}& &$r^*$ &$\rho^*$ &DIRECT/New &DIRECT{\it l}/New &DIRECT/New &DIRECT{\it l}/New \\
\hline
2 &$10^{-4}$ &.90 &.20  &2.88       &5.75  &2.70 &3.89\\
2 &$10^{-4}$ &.90 &.10  &1.77       &1.89  &1.68  &1.44 \\
\hline
3 &$10^{-6}$ &.66 &.20  &4.99       &5.31  &3.29  &5.34 \\
3 &$10^{-6}$ &.90 &.20  &$>$166.50  &4.87  &3.53  &6.86 \\
\hline
4 &$10^{-6}$ &.66 &.20  &$>$68.87   &8.18  &2.55  &6.61 \\
4 &$10^{-6}$ &.90 &.20  &$>$23.45   &6.75  &7.67  &8.56 \\
\hline
5 &$10^{-7}$ &.66 &.30  &$>$29.82   &5.31  &6.70  &11.13 \\
5 &$10^{-7}$ &.66 &.20 &$>$10.67 &$>$10.67  &5.83 &$>$12.82\\
\hline
\end{tabular}
\end{center}
\end{table}

Since DIRECT{\it l} at each iteration subdivides only one
hyperinterval among all hyperintervals with the same function
value, it was able to determine some other local minimizers (and
the global minimizer too) in the given maximal number of
trials~$T_{max}$. So, DIRECT{\it l} overcame the stagnation of the
search around the paraboloid vertex. But due to the locally-biased
character of DIRECT{\it l}, it spent too much trials exploring
various local minimizers which were not global. For this reason,
DIRECT{\it l} was unable to find the global minimizers of four
difficult five-dimensional functions.

In order to avoid stagnation of DIRECT near the paraboloid vertex
and to put DIRECT and DIRECT{\it l} in a more advantageous
situation, we shifted all generated functions, adding to their
values the constant 2. In such a way the value of each function at
the paraboloid vertex became equal to $2$ (and the global minimum
value~$f^*$ was increased by 2, i.e., became equal to $1$).
Results of numerical experiments with shifted GKLS classes
(defined in the rest by the same parameters) are reported in
Tables~\ref{table4} and \ref{table5}. Note that in this case
DIRECT has found the global solutions of all problems. DIRECT{\it
l} has not found the global minimizer of one five-dimensional
function. It can be seen from Tables~\ref{table4} and \ref{table5}
that also on these tests classes the new algorithm has manifested
its superiority with respect to DIRECT and DIRECT{\it l} in terms
of the number of generated trial points (Criterion C1).

Table~\ref{table6} summarizes (based on the data from Tables
\ref{table2}~--~\ref{table5}) results (in terms of Criterion C1)
of numerical experiments performed on 1600 test functions from
GKLS and shifted GKLS continuously differentiable classes. It
represents the ratio between the maximal number of trials
performed by DIRECT and DIRECT{\it l} with respect to the
corresponding number of trials performed by the new algorithm. It
can be seen from Table~\ref{table6} that the new method
outperforms both competitors significantly on the given test
classes when Criteria C1 and C2 are considered.

\begin{table}[t]
\caption{Average number of trial points for GKLS test functions
(Criterion C3).}
\begin{center} \scriptsize \label{table7}
\begin{tabular}{@{\extracolsep{\fill}}|c|c|c|c|r|r|r|c|c|}\hline
$N$ & $\Delta$ & \multicolumn{2}{c|}{Class} &DIRECT\hspace{1mm}
&DIRECT{\it l}\hspace{1mm} &New\hspace{2mm} &\multicolumn{2}{c|}{Improvement} \\
\cline{3-4} \cline{8-9} & & $r^*$ &$\rho^*$ & & & &DIRECT/New
&DIRECT{\it l}/New \\ \hline
2 &$10^{-4}$ &.90 &.20 &198.89   &292.79  &176.25 &1.13 &1.66 \\
2 &$10^{-4}$ &.90 &.10 &1063.78  &1267.07  &675.74 &1.57 &1.88 \\
\hline
3 &$10^{-6}$ &.66 &.20 &1117.70   &1785.73  &735.76 &1.52 &2.43 \\
3 &$10^{-6}$ &.90 &.20 & $>$42322.65  &4858.93  &2006.82 &$>$21.09 &2.42 \\
\hline
4 &$10^{-6}$ &.66 &.20 &$>$47282.89   &18983.55  &5014.13 &$>$9.43 &3.79 \\
4 &$10^{-6}$ &.90 &.20 &$>$95708.25   &68754.02  &16473.02 &$>$5.81 &4.17 \\
\hline
5 &$10^{-7}$ &.66 &.30 &$>$16057.46   &16758.44  &5129.85 &$>$3.13 &3.27 \\
5 &$10^{-7}$ &.66 &.20 &$>$217215.58   &$>$269064.35  &30471.83 &$>$7.13 &$>$8.83 \\
\hline
\end{tabular}
\end{center}
\end{table}

\begin{table}[t]
\caption{Average number of trial points for shifted GKLS test
functions (Criterion C3).}
\begin{center} \scriptsize \label{table8}
\begin{tabular}{@{\extracolsep{\fill}}|c|c|c|c|r|r|r|c|c|}\hline
$N$ & $\Delta$ & \multicolumn{2}{c|}{Class} &DIRECT &DIRECT{\it
l}\hspace{1mm} &New\hspace{2mm} &\multicolumn{2}{c|}{Improvement} \\
\cline{3-4} \cline{8-9} & & $r^*$ &$\rho^*$ & & & &DIRECT/New
&DIRECT{\it l}/New \\ \hline
2 &$10^{-4}$ &.90 &.20 &185.83   &249.25  &173.43  &1.07 &1.44 \\
2 &$10^{-4}$ &.90 &.10 &953.34   &1088.13  &609.36 &1.56 &1.79 \\
\hline
3 &$10^{-6}$ &.66 &.20 &951.04   &1434.33  &683.73  &1.39 &2.10 \\
3 &$10^{-6}$ &.90 &.20 &2226.36   &3707.85  &1729.55 &1.29 &2.14 \\
\hline
4 &$10^{-6}$ &.66 &.20 &7110.72   &14523.45  &4388.22  &1.62 &3.31 \\
4 &$10^{-6}$ &.90 &.20 &24443.60   &56689.06  &12336.56 &1.98 &4.60 \\
\hline
5 &$10^{-7}$ &.66 &.30 &5876.99   &10487.80 &4048.31  &1.45 &2.59 \\
5 &$10^{-7}$ &.66 &.20 &59834.38   &$>$182385.59  &19109.20 &3.13 &$>$9.54 \\
\hline
\end{tabular}
\end{center}
\end{table}

Let us now compare the three methods using Criteria C3 and C4.
Tables~\ref{table7} and~\ref{table8} report the average numbers of
trials performed during minimization of all 100 functions from the
same GKLS and shifted GKLS classes, respectively (Criterion~C3).
The columns ``Improvement'' of these tables represent the ratios
between the average numbers of trials performed by DIRECT and
DIRECT{\it l} with respect to the corresponding numbers of trials
performed by the new algorithm. The symbol~`$>$' reflects the
situation when not all functions of a class were successfully
minimized by the method under consideration in sense of
condition~(\ref{x*Found}). This means that the method stopped when
$T_{max}$ trials have been executed during minimization of several
functions of this particular test class. In these cases, the value
of $T_{max}$ equal to 1\,000\,000 was used in calculations of the
average value in~(\ref{avg_trials}), providing in such a way a
lower estimate of the average. As it can be seen from
Tables~\ref{table7} and~\ref{table8}, the new method outperforms
DIRECT and DIRECT{\it l} also on Criterion C3.

\begin{table}[!t]
\caption{Comparison between the new algorithm and DIRECT and
DIRECT{\it l} in terms of Criterion C4.}
\begin{center}\scriptsize \label{table9}
\begin{tabular}{@{\extracolsep{\fill}}|c|c|c|c|c|c|c|c|}\hline
$N$ & $\Delta$ & \multicolumn{2}{c|}{Class} &
\multicolumn{2}{c|}{GKLS} & \multicolumn{2}{c|}{Shifted GKLS} \\
\cline{3-8}& &$r^*$ &$\rho^*$ &DIRECT\,:\,New
&DIRECT{\it l}\,:\,New &DIRECT\,:\,New &DIRECT{\it l}\,:\,New \\
\hline
2 &$10^{-4}$ &.90 &.20  &61\,:\,39    &52\,:\,47   &61\,:\,38   &54\,:\,46  \\
2 &$10^{-4}$ &.90 &.10  &36\,:\,64    &23\,:\,77   &37\,:\,63   &23\,:\,77  \\
\hline
3 &$10^{-6}$ &.66 &.20  &66\,:\,34    &54\,:\,46   &65\,:\,35   &62\,:\,38  \\
3 &$10^{-6}$ &.90 &.20  &58\,:\,42    &51\,:\,49   &56\,:\,44   &54\,:\,46  \\
\hline
4 &$10^{-6}$ &.66 &.20  &51\,:\,49    &37\,:\,63   &50\,:\,50   &44\,:\,56  \\
4 &$10^{-6}$ &.90 &.20  &47\,:\,53    &42\,:\,58   &46\,:\,54   &43\,:\,57  \\
\hline
5 &$10^{-7}$ &.66 &.30  &66\,:\,34    &26\,:\,74   &67\,:\,33   &42\,:\,58  \\
5 &$10^{-7}$ &.66 &.20  &34\,:\,66    &27\,:\,73   &32\,:\,68   &32\,:\,68  \\
\hline
\end{tabular}
\end{center}
\end{table}

Finally, results of comparison between the new algorithm and its
two competitors in terms of Criterion C4 are reported in
Table~\ref{table9}. This table shows how often the new algorithm
was able to minimize each of 100 functions of a class with a
smaller number of trials with respect to DIRECT or DIRECT{\it l}.
The notation `$p$\,:\,$q$' means that among~100 functions of a
particular test class there are $p$ functions for which DIRECT (or
DIRECT{\it l}) spent less function trials than the new algorithm
and $q$ functions for which the new algorithm generated fewer
trial points with respect to DIRECT (or DIRECT{\it l}) ($p$ and
$q$ are from~(\ref{winDIRECT}) and~(\ref{winNEW}), respectively).
For example, let us compare the new method with DIRECT{\it l} on
the GKLS two-dimensional class with parameters $r^* = 0.90$,
$\rho^* = 0.20$ (see Table~\ref{table9}, the cell `52\,:\,47' of
the first line). We can see that DIRECT{\it l} was better (was
worse) than the new method on $p=52$ ($q=47$) functions of this
class, and one problem was solved by the two methods with the same
number of trials.

It can be seen from Table~\ref{table9} that DIRECT and DIRECT{\it
l} behave better than the new algorithm with respect to
Criterion~C4 when simple functions are minimized (we remind that
for each problem dimension the first class is simpler than the
second one). For example, for the difficult GKLS two-dimensional
class and DIRECT{\it l} we have `23\,:\,77' instead of `52\,:\,47'
for the simple class. If a more difficult test class is taken, the
new method outperforms its two competitors (see the second --
difficult~-- classes of the dimensions $N=2$, 4, and 5 in
Table~\ref{table9}). For the three-dimensional classes DIRECT and
DIRECT{\it l} were better than the new method (see
Table~\ref{table9}). This happens because the second
three-dimensional class (even being more difficult than the first
one because the number $q$ has increased in all the cases)
continues to be too simple. Thus, since the new method is oriented
on solving difficult multidimensional multiextremal problems, more
hard objective functions are presented in a test class, more
pronounced is advantage of the new algorithm.

\section{A brief conclusion}

The problem of global minimization of a multidimensional
``black-box'' function satisfying the Lipschitz condition over a
hyperinterval with an unknown Lipschitz constant has been
considered in this paper. A new algorithm developed in the
framework of diagonal approach for solving the Lipschitz global
optimization problems has been presented. In the algorithm, the
partition of the admissible region into a set of smaller
hyperintervals is performed by a new efficient diagonal partition
strategy. This strategy allows one to accelerate significantly the
search procedure in terms of function evaluations with respect to
the traditional diagonal partition strategies. A new technique
balancing usage of the local and global information has been also
incorporated in the new method.

In order to calculate the lower bounds of $f(x)$ over
hyperintervals, possible estimates of the Lipschitz constant
varying from zero to infinity are considered at each iteration of
the algorithm. The procedure of estimating the Lipschitz constant
evolves the ideas of the popular method DIRECT
from~\cite{Jones:et:al.(1993)} to the case of diagonal algorithms.
The `everywhere dense' convergence of the new algorithm has been
established. Extensive numerical experiments executed on more than
1600 test functions have demonstrated a quite satisfactory
performance of the new algorithm with respect to
DIRECT~\cite{Jones:et:al.(1993)} and DIRECT{\it l}
\cite{Gablonsky(2001), Gablonsky&Kelley(2001)} when hard
multidimensional functions are minimized.

\vspace{5mm}

\bibliographystyle{siam}
\bibliography{directADC_Rev}

\end{document}